\documentclass[11pt,leqno,dvipdfmx]{amsart}
\usepackage{amsmath,epsfig,graphicx,color}
\usepackage{mathrsfs,amssymb, amsfonts, latexsym, mathtools}
\usepackage{enumerate}
\usepackage{amsaddr}
\definecolor{darkblue}{rgb}{.2, 0.2,.8}
\definecolor{carageen}{rgb}{0,0.5,0.3}
\definecolor{darkred}{rgb}{.8, .1,.1}

\newtheorem{lemma}{Lemma}[section]

\newtheorem{theorem}[lemma]{Theorem}

\newtheorem{proposition}[lemma]{Proposition}
\newtheorem{definition}[lemma]{Definition}
\newtheorem{corollary}[lemma]{Corollary}
\newtheorem{example}[lemma]{Example}
\newtheorem{exercise}[lemma]{Exercise}
\newtheorem{remark}[lemma]{Remark}
\newtheorem{fig}[lemma]{Figure}
\newtheorem{tab}[lemma]{Table}

\newcommand{\bth}{\begin{theorem}}
\newcommand{\ethe}{\end{theorem}}

\newcommand{\bre}{\begin{remark}\em }
\newcommand{\ere}{\end{remark}}

\newcommand{\ble}{\begin{lemma}}
\newcommand{\ele}{\end{lemma}}

\newcommand{\bde}{\begin{definition}}
\newcommand{\ede}{\end{definition}}
\newcommand{\bco}{\begin{corollary}}
\newcommand{\eco}{\end{corollary}}

\newcommand{\bpr}{\begin{proposition}}
\newcommand{\epr}{\end{proposition}}

\newcommand{\bexer}{\begin{exercise}}
\newcommand{\eexer}{\end{exercise}}

\newcommand{\bexam}{\begin{example}}
\newcommand{\eexam}{\end{example}}

\newcommand{\bfi}{\begin{fig}}
\newcommand{\efi}{\end{fig}}

\newcommand{\btab}{\begin{tab}}
\newcommand{\etab}{\end{tab}}

\newcommand{\beao}{\begin{eqnarray*}}
\newcommand{\eeao}{\end{eqnarray*}\noindent}

\newcommand{\beam}{\begin{eqnarray}}
\newcommand{\eeam}{\end{eqnarray}\noindent}

\newcommand{\beqq}{\begin{equation}}
\newcommand{\eeqq}{\end{equation}\noindent}

\newcommand{\bce}{\begin{center}}
\newcommand{\ece}{\end{center}}

\newcommand{\barr}{\begin{array}}
\newcommand{\earr}{\end{array}}

\newcommand{\vague}{\stackrel{\lower0.2ex\hbox{$\scriptscriptstyle
                    \it{v} $}}{\rightarrow}}
\newcommand{\weak}{\stackrel{\lower0.2ex\hbox{$\scriptscriptstyle
                    \it{w} $}}{\rightarrow}}
\newcommand{\what}{\stackrel{\lower0.2ex\hbox{$\scriptscriptstyle
                    \it{\hat{w}} $}}{\rightarrow}}

\newcommand{\bdis}{\begin{displaymath}}
\newcommand{\edis}{\end{displaymath}\noindent}

\newcommand{\N}{\mathbb{N}}
\newcommand{\R}{\mathbb{R}}

\newcommand{\ov}{\overline}
\newcommand{\wt}{\widetilde}
\newcommand{\wh}{\widehat}
\newcommand{\vep}{\varepsilon}

\newcommand{\cals}{{\mathcal S}}
\newcommand{\call}{{\mathcal L}}

\newcommand{\idr}{${\rm ID}(\R)$}
\newcommand{\id}{{\rm ID}(\R)}

\newcommand{\E }{{\mathbb E}}

\allowdisplaybreaks

\begin{document}
\today
\bibliographystyle{plain}
\title[Subexponentialiy of densities of infinitely divisible distributions]{Subexponentiality of densities of infinitely divisible distributions}

\author[M. Matsui]{Muneya Matsui}
\address{Department of Business Administration, Nanzan University, 18
Yamazato-cho, Showa-ku, Nagoya 466-8673, Japan.}
\email{mmuneya@gmail.com}

\begin{abstract}{
We show the equivalence of three properties for an infinitely divisible distribution: 
the subexponentiality of the density, the subexponentiality of the density of its L\'evy 
measure and the tail equivalence between the density and its L\'evy measure density,  
under monotonic-type assumptions on the L\'evy 
measure density. 
The key assumption is that tail of the L\'evy measure density is 
asymptotic to a non-increasing function
or is almost decreasing. 
Our conditions are natural and cover a rather wide class of infinitely divisible distributions. 
Several significant properties for analyzing the subexponentiality of densities have been derived such as 
closure properties of [\,convolution, convolution roots and asymptotic equivalence\,] and the factorization property. 
Moreover, we illustrate that the results are applicable for developing the 
statistical inference of subexponential infinitely divisible distributions which are absolutely continuous.
} 
\end{abstract}
\maketitle

\section{Introduction}
Studies on the subexponentiality of infinitely divisible distributions have been initiated by Embrecht et al. 
\cite{Embrechts:Goldie:Veraverbeke:1979}, where the subexponentiality of one-sided distributions was completely characterized.  
Pakes \cite{Pakes:2004} extended the result into distributions on the real line. 
More general $\gamma$-subexponentiality $\gamma\ge 0$ (see e.g \cite[p.369]{Watanabe:2008}) 
has intensively investigated by Embrechts and Goldie \cite{Embrechts:Goldie:1982}, Pakes \cite{Pakes:2004} and Watanabe \cite{Watanabe:2008} 
(see a comprehensive literature in the introduction in \cite{Watanabe:2008}). 

However, on the subexponentiality of densities of infinitely divisible distributions 
there are only a few works.
Watanabe and Yamamuro \cite{Watanabe:Yamamuro:2010} investigated 
the class of self-decomposable distributions, and 
Watanabe \cite{Watanabe:2020} studied the subexponential densities on the half-line. Shimura and Watanabe \cite{shimura:watanabe:2022} 
treated the compound Poisson case on the positive half. 
As stated in Watanabe \cite{Watanabe:2020} ``
the subexponentiality of a density is a stronger and more difficult property than the subexponentiality of a distribution.''
 Besides, we treat two-sided distributions which 
are harder to handle than one-sided ones. 
The results have been applied to characterize the tail asymptotics for the density 
of the supremum of a random walk, which is closely related with classical ruin theory and queuing theory 
(\cite[Section 5]{Foss:Korshunov:Zachary:2013}, \cite[Section 4]{shimura:watanabe:2022}).

One of our motivations here is an application in statistics.  
Infinitely divisible distributions provide finite dimensional distributions of L\'evy processes, 
the processes which have found numerous applications in meteorology, seismology, telecommunications, finance and insurance,
and still attract a lot of attention (see \cite{bmr:2001}). Thus, the related statistical methods have been 
intensively studied. In applications such as statistical modelings or statistical estimations of L\'evy processes, 
densities are often more convenient than distributions to handle.
Therefore, further studies on tail properties of densities such as subexponentiality 
are desirable.   
Indeed, the tail condition is crucial for asymptotics of various estimation methods (see the argument in Section \ref{sec:application:statistics}).

In this paper we characterize the subexponentiality of densities of 
infinitely divisible distributions on the whole real line. 
We establish the tail equivalence between the absolutely continuous part of an infinitely divisible distribution
and the density of the corresponding L\'evy measure. Furthermore, we show that  
this equivalence implies that the equivalence in subexponentiality, and vice versa. 
Our key is to assume a kind of monotonic property  
on the L\'evy measures, that is 
``asymptotic to a non-increasing function'' (abbreviated by a.n.i.) property 
or ``almost decreasing'' (abbreviated by al.d.) property.  
The former assumption is a bit stronger, but we could derive stronger results. 
Notice that the regularly varying functions with negative indices satisfy the a.n.i. property (\cite[p.23]{Bingham:Goldie:Teugels:1989}). 
Moreover, these two properties are covered by a rather wide class of L\'evy measures, and indeed 
both are shared by all self-decomposable distributions.

Our strategy is to loosen monotonic-type assumptions as much as possible, while keeping 
the equivalence of the three properties. Our solution in the compound Poisson case is the al.d. property, 
which is different from the one-sided case, where we do not need any monotonic-type assumptions. 
We derive several significant tools under this property for analyzing tail behaviors, 
which are used to prove the main result of the three equivalent relation. 
Notice that if the al.d. condition is violated , one can make an example such that 
the equivalence of subexponentiality does not hold in the two-sided case (see Section \ref{sec:cp}). 

For the infinitely divisible case, we assume the 
a.n.i. condition, a stronger condition than al.d., since the al.d. condition is not enough for our purpose. 
Indeed, the proof of the three equivalent relation, which is the main result, 
 is rather different form that of compound Poisson case. Our idea is to connect 
the subexponentiality of density and the local subexponentiality, which 
is an intermediate notion between the density and distribution. 
Different from the distribution, if the density is treated, there are harder gaps between 
the compound Poisson and infinitely divisible cases. 


We apply our results to the consistency proof of the maximum likelihood estimation (MLE for abbreviation) for an 
absolutely continuous infinitely divisible distribution. 
Usually an explicit expression for the density is unavailable for this class, while 
properties of the density such as boundedness and tail behavior (and sometimes continuity) 
are crucial in both the definition and asymptotics of MLE.  
Our proof depends only on the L\'evy density and we do not touch 
the genuine density or distribution. 
Therefore, by our results we could extend the scope of MLE to a rather wide subclass of infinitely divisible distributions 
beyond particular ones with explicit densities. We believe that our results would be useful tools for other estimation methods than MLE. 

In Section \ref{sec:preliminaries}, notation and definitions are formulated. 
We state main results for infinitely divisible distributions 
in Section \ref{sec:main:results} together with 
the closure/factorization properties of the subexponential density 
under assumption of the a.n.i. or al.d. condition. The compound Poisson case treated 
in Section \ref{sec:cp}, where the al.d. condition is exploited. 
In Section \ref{sec:application:statistics}, a statistical application is provided. 
The proofs of the main and necessary auxiliary results are given in Section 
\ref{pf:main:ID}. The proofs related with the convolution root, which are not directly 
related with the main results, are summarized in Appendix \ref{append:proofs}.  

 \section{Preliminaries}
\label{sec:preliminaries}
Let $F,G,H$ be probability distribution functions on $\R$ and denote by $F\ast G$ the convolution of $F$ and $G$:
\[
 F\ast G(x)=\int_{-\infty}^\infty F(x-y)G(d y)
\] 
and denote by $F^{\ast n}$ the $n$th convolutions with itself. The tail probability of $F$ is denoted by $\ov F(x)=1-F(x)$. 
Let $f,g,h$ be the corresponding probability density functions on $\R$ and we use the same notations for the convolution as those for distributions, e.g. 
\[
 f\ast g(x)=\int_{-\infty}^\infty f(x-y)g(y) d y \quad \text{or}\quad f^{\ast n}(x) \quad \text{for the $n$th convolution}. 
\]
We say that $F$ on $\R$ is long-tailed, denoted by $F\in \call$, if $\ov F(x)>0$ for all $x$ and 
\[
 \lim_{x\to\infty} \ov F(x+y)/ \ov F(x)=1\quad \text{for any fixed}\quad y>0. 
\]
In addition we call that $F$ is subexponential on $\R$, denoted by $F\in  \cals$, if $F\in \call$ and  
\begin{align}
\label{condi:subexp:posi}
 \lim_{x\to \infty} \ov{F^{\ast 2}}(x)/ \ov F(x)=2. 
\end{align}
The class $\cals$ was introduced by \cite{Chistyakov:1964} and it is known that $\cals$ includes regularly varying functions. 
It should be noted that if $F$ is a distribution on $\R_+=:[0,\infty)$, then the condition \eqref{condi:subexp:posi} solely implies 
$F\in\cals$ and we do not need the assumption $F\in \call$, since $F\in \cals$ automatically satisfies $F \in\call$ 
(cf. \cite[Lemma 3.2]{Foss:Korshunov:Zachary:2013}). 
Throughout the paper, for functions $\alpha,\beta:\R \to \R_+$, $\alpha(x) \sim \beta(x)$ 
means that $\lim_{x\to\infty}\alpha(x)/\beta(x)\to 1$. 

In this paper we study the corresponding characteristics for densities. 
\begin{definition}
$(\mathrm{i})$ The density $f$ of $F$ is $($right-side$)$ long-tailed, denoted by $f \in \call$, if there exists $x_0>0$ such that 
$f(x)>0,\,x\ge x_0$ and for any fixed $y>0$ $f(x+y)\sim f(x)$. \\
$(\mathrm{ii})$ The density $f$ of $F$ is $($right-side$)$ subexponential on $\R$, denoted by $\cals$, 
if $f\in \call$ and $f^{\ast 2}(x) \sim 2f(x)$. \\
$(\mathrm{iii})$ The density $f$ of $F$ is weakly (right-side) subexponential on $\R$, denoted by $\cals_+$, 
if $f\in \call$ and the function $f_+(x)={\bf 1}_{\R_+}(x) f(x) / \ov F(0),\,x \in \R$ is subexponential, i.e. $f_+\in \cals$. 
\end{definition}
Here $f_+$ is the density of the conditional distribution $F_+$ of $F$ on $\R_+$. 
For a distribution $F$, it is known that 
\[
F\in \cals\ \Leftrightarrow\ F_+ \in \cals\ \Leftrightarrow\ F^+ \in \cals, 
\] where 
$F^+$ is the distribution given by 
$F^{+}(x)=F(x)$ for $x\ge 0$ and $F^{+}(x)=0$ for $x<0$ (see Corollary 2.1 of \cite{Pakes:2004}, Lemma 3.4 of \cite{Foss:Korshunov:Zachary:2013}).  
However, for a probability density $f$ the situation is different, i.e. 
unless the support of $f$ is bounded below, we could not have $f\in \cals \Leftrightarrow f \in \cals_+$ without additional 
conditions (\cite[p.83]{Foss:Korshunov:Zachary:2013}). Therefore, we assume one of 
the following two assumptions, under which $f\in \cals \Leftrightarrow f \in \cals_+$
(\cite[Lemma 4.13]{Foss:Korshunov:Zachary:2013}), and which are key tools 
in this paper. 
\begin{definition}
$(\mathrm{i})$ We say that a density $f:\R\to \R_+$ is asymptotic to a non-increasing function $($a.n.i. for short$)$ if 
$f$ is locally bounded and positive on $[x_0,\infty)$ for some $x_0>0$, and 
\begin{align}
\label{eq:def:ani}
 \sup_{t\ge x} f(t)\sim f(x)\quad \text{and}\quad \inf_{x_0\le t \le x} f(t)\sim f(x). 
\end{align}
$(\mathrm{ii})$
We say that a density $f:\R\to\R_+$ is almost decreasing $($al.d. for short$)$ if there exists $x_0>0$ and $K>0$ such that
\begin{align*}
 f(x+y) & \le Kf(x),\quad \text{for all}\ x>x_0,\,y>0. 
\end{align*} 
\end{definition}
Notice that the al.d. property includes the a.n.i. property, and the latter is 
satisfied by the regularly varying functions with negative indices \cite[p.23]{Bingham:Goldie:Teugels:1989}. 

We will investigate properties of the above sort, particularly on infinitely divisible distributions $\mu$ on $\R$. 
The characteristic function (ch.f.) of $\mu$ is 
\begin{align}
\label{def:chf:idr}
 \wh \mu (z) = \exp \Big\{
\int_{-\infty}^\infty 
(e^{izy}-1-i zy{\bf 1}_{\{ |y|\le 1 \}} ) \nu (dy) +iaz -\frac{1}{2} b^2 z^2
\Big\},
\end{align}
where $a \in \R,\,b\ge 0$ and $\nu$ is the L\'evy measure satisfying $\nu(\{0\})=0$ and $\int_{-\infty}^\infty (1 \wedge x^2) \nu (dx) <\infty$.
Throughout this paper, we always assume that the L\'evy measure $\nu$ of $\mu$ has a density, and 
we denote by \idr\ the class of all infinitely divisible distributions on $\R$.  

\section{Main results}
\label{sec:main:results}
We separate the cases depending on weather $\nu(\R)<\infty$ or $\nu(\R)=\infty$. 
The former implies that $\mu$ is a compound Poisson plus Gaussian 
(e.g. \cite[Ch.$\mathrm{IV}$, Theorem 4.1.8]{steutel:VanHarn:2003}, \cite[Lemma 2.13]{Kuprianou:2006}). 
Since we always assume a density for the L\'evy measure, 
the latter implies that the purely non-Gaussian part of $\mu$ is absolutely continuous (e.g. \cite[Theorem 27.7]{sato:1999} with $l=1$). 
Note that we use notation $g$ also for the (non-proper) density of a L\'evy measure. 
\begin{theorem}
\label{thm:id:cp}
Let $\mu \in \id$ with $\nu(dx)=g(x) dx$ such that $\nu(\R)<\infty$.  
Denote the non-Gaussian part $\mu'$, which is a $\gamma$-shifted compound Poisson given by  
\[
 \mu'(dx)= e^{-\lambda} \delta_\gamma(dx) + (1-e^{-\lambda}) f(x-\gamma) dx, \quad \gamma\in \R, 
\]
where $\delta_\gamma$ is Dirac measure at $\gamma$, $f$ is a proper density and $\lambda>0$ is the Poisson parameter. 
Then the following are equivalent. 
\begin{align*}
(\mathrm{i})\quad & \quad f\in \cals_+\ \text{and}\ f\ \text{is al.d.} \\
(\mathrm{ii})\quad & \quad g\in \cals_+\ \text{and}\ g\ \text{is al.d.}  \\
(\mathrm{iii})\quad & \quad g\in\call,\ g\ \text{is al.d.} \ \text{and}\ \lim_{x\to \infty} f(x)/g(x) =\lambda/(1-e^{-\lambda}). 
\end{align*}
\end{theorem}
Theorem \ref{thm:id:cp} directly follows from the compound Poisson case (Theorem \ref{thm:compundpoi}).
\begin{theorem}
\label{theorem:ID}
Let $\mu \in \id$ 
with $\nu(dx)=g(x) dx$ such that $\nu(\R)=\infty$. 
Let $f_0(x)$ be a density of $\mu_0 \in \id$ with $a=b=0$ and $\nu(dx)={\bf 1}_{\{|x|\le 1\}}g(x)dx$.
Suppose that 
$g_1(x)={\bf 1}_{\{x>1\}} g(x)/\nu((1,\infty))$ is bounded, and there exists $\gamma>0$ such that 
\begin{align}
\label{exp:limit:trancated}
 \lim_{x\to \infty} e^{\gamma x} f_0(x)=0.
\end{align}
For a density $f$ of $\mu$ we consider the following properties. 
\begin{align*}
(\mathrm{i})\quad & f\in \cals_+\quad \text{and}\quad f\ \text{is al.d.} \\
(\mathrm{ii})\quad & g_1\in \cals_+ \\
(\mathrm{iii})\quad & g_1 \in \call\quad \&\quad \lim_{x\to \infty} f(x)/g_1(x) = \nu((1,\infty)).  
\end{align*}
$(a)$ If $g$ is a.n.i., then we can choose $f$ such that $(\mathrm{i})$, $(\mathrm{ii})$ and $(\mathrm{iii})$ are equivalent.  \\
$(b)$ If $g$ is al.d., then we can choose $f$ such that $(\mathrm{ii})\Leftrightarrow(\mathrm{iii})$ implies $(\mathrm{i})$.
\end{theorem}
The proof is given in Section \ref{pf:main:ID}. 
Obviously $\cals_+$ of the two theorems is replaced with $\cals$. 

Since the a.n.i. property includes regular variation, the following is immediate. 
\begin{corollary}
\label{corollary:ID}
Assume the same conditions as those of Theorem \ref{theorem:ID}
before three properties $(\mathrm{i})$, $(\mathrm{ii})$ and $(\mathrm{iii})$. 
Then, $[$$g_1(x)$ is a.n.i. and 
we can choose a density $f$ of $\mu$ such that it is regularly varying$]$ if and only if 
$g$ is regularly varying, and in this case $f(x)\sim g(x)$. 
\end{corollary}
\begin{remark}
$(\mathrm{i})$ 
Theorem \ref{theorem:ID} holds regardless of values of $b$, and thus
 Gaussian density is convoluted in $f$ when $b> 0$. Generally it is difficult to separate non-Gaussian part 
from $f$, since it also includes the small jump part from $g(x){\bf 1}_{\{|x|\le 1\}}$ as a convoluted element, which is not heavy tail. \\
$(\mathrm{ii})$ Theorem \ref{theorem:ID} covers many important subclasses of\ \idr, such as $s$ self-decomposable distributions 
whose L\'evy densities $g$ are non-increasing on $(0,\infty)$ and non-decreasing on $(-\infty,0)$ $($Jurek \cite{Jurek:1985}$)$. 
Indeed, it includes all self-decomposable cases of \cite{Watanabe:Yamamuro:2010}. \\
$(\mathrm{iii})$ Notice that the al.d or a.n.i. property of $g_1 \in \cals$ does not imply 
that $g_1$ is regularly varying. Recall from \cite[p.86]{Foss:Korshunov:Zachary:2013} $($cf. \cite[p.1042]{Watanabe:Yamamuro:2010}$)$ that 
the density of the standard log normal distribution or the Weibull distribution with parameter $\alpha\in(0,1)$ is 
 subexponential and a.n.i., but it is not regularly varying. Both distributions are known to be self-decomposable 
$($cf. \cite[p.360, p.414]{steutel:VanHarn:2003}$)$. 
\\
$(\mathrm{iv})$ We could easily find non self-decomposable examples which are covered by 
Theorem \ref{theorem:ID}. Indeed, assume that the tail of L\'evy density $g$ is given by  
the density of the standard log normal distribution or the Weibull distribution with parameter $\alpha\in(0,1)$, and further assume that 
$g$ is not monotone. Then, $\mu \in \id$ with this L\'evy density $g$ 
is not self-decomposable, while $g_1\in \cals$. 
\end{remark}

We could remove several conditions in Theorem \ref{theorem:ID} by assuming the absolute integrability of $\wh \mu$ (cf. \cite[Theorem 2]{Watanabe:2020}). 
In our case we require a stronger condition, since we treat the two sided case. Define a spectrally positive version $\mu_+$ by  
\begin{align}
\label{def:id:spectrally:positive}
 \wh \mu_+(z)= \exp\Big\{
\int_0^\infty (e^{izx}-1-izx{\bf 1}_{\{ 0 < y \le 1 \}} )\nu(dx)
\Big\}. 
\end{align}
and 
assume the absolute integrability of $\wh \mu_+$.  
Although the proof is made by a minor change to that of Theorem \ref{theorem:ID}, the result is convenient in applications 
and we state the result as a theorem.   
\begin{theorem}
\label{theorem:ID:abs}
 Let $\mu\in {\rm ID}(\R)$ with and $\nu(dx)=g(x) dx$ such that $g_1(x)$ is bounded. Suppose that 
$\int_{-\infty}^\infty |\wh \mu_+(z)|dz<\infty$, which implies $\int_{-\infty}^\infty |\wh \mu(z)|dz<\infty$, 
so that $\mu$ has a bounded continuous density $f$. Then 
the following relations hold between the properties $(\mathrm{i})$, $(\mathrm{ii})$ and $(\mathrm{iii})$ of Theorem \ref{theorem:ID}. \\
$(a)$ If $g$ is a.n.i., then we can choose $f$ such that $(\mathrm{i})$, $(\mathrm{ii})$ and $(\mathrm{iii})$ are equivalent.  \\
$(b)$ If $g$ is al.d., then we can choose $f$ such that $(\mathrm{ii})\Leftrightarrow(\mathrm{iii})$ implies $(\mathrm{i})$.
\end{theorem}
The proof is given in the end of Section \ref{pf:main:ID}.

\subsection{Known properties of $f$ in $\call$, $\cals_+$ or $\cals$ }
We introduce known or easily-derived properties
of $f$ in $\call$, $\cals_+$ or $\cals$. 
Since we handle a compound Poisson distribution, 
which is not absolutely continuous, 
we introduce a generalized density,
\begin{align}
\label{def:gene:density}
 \wt f(x)= p \delta (x)+q f(x),\quad p+q=1,\quad p,q\ge 0,
\end{align}
where $f$ is a proper density and $\delta(x)$ is the Dirac delta function (see \cite{Khuri:2004} for treatment of $\delta(x)$). 
Here $f$ is an ordinary function and it does not 
include $\delta(x)$. 
Notice that by direct calculations
$f\in \call \Leftrightarrow \wt f \in \call$ and $f\in \cals \Leftrightarrow \wt f \in \cals$  
clearly hold. Moreover $f\in \cals_+ \Leftrightarrow \wt f \in \cals_+$ follows, which is equivalent to 
$f\in \cals \Leftrightarrow \wt f \in \cals$ if $\wt f$ is on $\R_+$. 
The properties of $\wt f$ rely heavily on the following crucial concept by \cite[Definition 2.18]{Foss:Korshunov:Zachary:2013}, 
which the class $\call$ enjoys and which we will use frequently. 
\begin{definition}
Given a strictly positive non-decreasing function $\alpha$, 
an ultimately positive function $\beta$ is called $\alpha$-insensitive if 
\[
 \sup_{|y| \le \alpha(x)} |\beta (x+y)-\beta (x)|=o(\beta (x))\quad \text{as}\ x\to\infty, \ \text{uniformly in}\quad |y| \le \alpha(x). 
\]
\end{definition}
The following is known for $f$, but easily extended for $\wt f$ with slight modifications in the proof. 
\begin{lemma}
\label{lem:a2-1}
$(\mathrm{i})$ \cite[Lemma 2.19]{Foss:Korshunov:Zachary:2013}. 
 Let $\wt f \in \call$ and then there exists a function $\alpha$ such that $\alpha(x)\to \infty$ as 
$x\to \infty$ and $\wt f$ is $\alpha$-insensitive. \\
$(\mathrm{ii})$ \cite[Proposition 2.20]{Foss:Korshunov:Zachary:2013}. 
Given a finite collection of $\wt f_1,\ldots,\wt f_n \in \call$ we may choose a single function $\alpha$ increasing to infinity w.r.t. which 
each of functions $\wt f_i$ is $\alpha$-insensitive. \\
$(\mathrm{iii})$ \cite[Theorem 4.2, Corollary 4.5]{Foss:Korshunov:Zachary:2013}. Let $\wt f \in \call$, then 
\begin{align*}
 \liminf_{x\to\infty} \wt f\ast \wt g (x)/\wt f(x) \ge 1.
\end{align*}
Moreover, $\wt f^{\ast n} \in \call$ and 
\begin{align*}
 \liminf_{x\to\infty} \wt f^{\ast n}(x)/\wt f(x) \ge n. 
\end{align*}
$(\mathrm{iv})$ \cite[Theorem 4.3]{Foss:Korshunov:Zachary:2013}. Let $\wt f,\,\wt g \in \call$, then $\wt f\ast \wt g\in \call$. \\
$(\mathrm{v})$ cf. \cite[Proofs of Lemmas 4.12 and 4.13]{Foss:Korshunov:Zachary:2013}. 
Let $\wt f \in S$ and is $\alpha$-insensitive such that $\alpha(x)<x/2$ and $\alpha(x)\to \infty$ as $x\to \infty$. Then 
\begin{align}
\label{condi:subexponential:real}
 \int_{-\infty}^{-\alpha(x)} \wt f(x-y) \wt f(y)dy =o(\wt f(x))\quad \text{and}\quad \int_{\alpha(x)}^{x-\alpha(x)} \wt f(x-y)\wt f(y)dy = o(\wt f(x)). 
\end{align}
$(\mathrm{vi})$ \cite[Theorem 4.8]{Foss:Korshunov:Zachary:2013}. If $\wt f \in \cals_+$ and $\wt g(x)\sim c \wt f(x)$ with $c>0$, then $\wt g\in\cals_+$.
\end{lemma}
We only give the proof for $(\mathrm{v})$. 
\begin{proof}[Proof of {\rm Lemma \ref{lem:a2-1} $(\mathrm{v})$}]
We take an insensitive function $\alpha$ for $\wt f$ and write 
\begin{align*}
 \frac{\wt f^{\ast 2}(x)}{\wt f(x)}
&= 2 \int_{-\infty}^{-\alpha(x)} \frac{\wt f(x-y)\wt f(y)}{\wt f(x)}dy + 2 \int_{-\alpha(x)}^{\alpha(x)} 
\frac{\wt f(x-y) \wt f(y)}{\wt f(x)}dy + \int_{\alpha(x)}^{x-\alpha(x)} \frac{\wt f(x-y) \wt f(y)}{\wt f(x)}dy \\
&=:2I_1(x)+2I_2(x)+I_3(x). 
\end{align*} 
By the property of $\delta(x)$ and $\wt f\in \call \Leftrightarrow f \in \call$, we have 
\[
 I_2(x)=p+q \int_{-\alpha(x)}^{\alpha(x)} f(x-y)f(y)/f(x)dy \to 1,
\]
as $x\to \infty$. Thus other integrals $I_1,I_3\ge 0$ should converge to zero. 
\end{proof}
The next result is Kesten's type bound for densities. 
\begin{lemma}[{\cite[Theorem 4.11]{Foss:Korshunov:Zachary:2013} $\&$ \cite[Theorem 2]{Finkelshtein:Tkachov:2018}}]
\label{lem:kestenbound}
Let $f\in \cals_+$ be bounded. 
Suppose that $f$ is a density on $\R_+$ or that $f$ is al.d. 
Then for any $\vep \in(0,1)$, there exist $C_\vep>0$ and $x_\vep>0$ such that 
\[
 f^{\ast n} (x) \le C_\vep (1+\varepsilon)^n f(x),\quad x>x_\vep,\,n\in \N. 
\] 
\end{lemma}

\subsection{New properties of $f$ in $\call$, $\cals_+$ or $\cals$ with al.d/a.n.i. condition}
We close this section with the following results, which are crucial for proving main theorems and 
which are significant in themselves for analyzing tail behavior of densities under the al.d. or a.n.i. assumption.  
Proofs of the results needed for the main results are given in Section \ref{subsec:pr:other:section3}, 
while that Theorem \ref{prop:a3-5} (convolution root)
 is given in Appendix \ref{append:proofs}, which is not directly related with the main theorems.  
The first two lemmas are easy properties of al.d. and a.n.i. respectively.
\begin{lemma}
\label{lem:equiv:ani:function}
 Let $f$ be a density with the a.n.i. property. Then 
\begin{align}
\label{lem:ani:function}
\text{There exists a positive non-increasing function}\ \alpha\ \text{such that} \lim_{x\to\infty} f(x)/\alpha (x)=1. 
\end{align}
Conversely, if $f(x)\to 0$ as $x\to \infty$, then \eqref{lem:ani:function} implies the a.n.i. property of $f$. 
In particular, $f\in\call$ with the condition \eqref{lem:ani:function} implies that $f$ is a.n.i. 
\end{lemma}
\begin{lemma}
\label{lem:asymp:eqive:ald}
Suppose that $\wt f$ is al.d. and $\lim_{x\to\infty} \wt f(x)/\wt g(x)=c$ with some $c>0$. Then $\wt g$ is al.d. 
\end{lemma}
\begin{lemma}[Asymptotic equivalence]
\label{lem:a3-2}
Suppose $\wt f\in \cals$ and 
\begin{align}
\label{asmp:lem:a3-2}
\lim_{x\to\infty} \wt f(x)/\wt g(x)=c\quad  \text{for some} \quad c \in (0,\infty). 
\end{align} 
If $\wt g$ or $\wt f$ is al.d., or $\wt g(-x)=O(\wt f(-x))$ then $\wt g \in \cals$. 
\end{lemma}
\begin{proposition}[Convolution and factrization]
\label{prop:a3-3}
Let $h=f\ast \wt g$ be the convolution of a density $f$ and a generalized density $\wt g$.\\
$(\mathrm{i})$ Let $f\in \cals_+$ and $\wt g(x)=o(f(x))$. If $f$ is al.d.
or $g(-x)=O(f(-x))$ as $x\to\infty$, then $h\in \cals_+$ and $h(x) \sim f(x)$. \\
$(\mathrm{ii})$ Let $h \in \cals_+$ and $\wt g(x)=o(h(x))$. 
If $f$ is a.n.i. and $[$$h$ is a.n.i. or $ \wt g(x)=o(f(x))$$]$, then $f\in \cals_+$ and moreover $h(x)\sim f(x)$.  
\end{proposition}
\begin{theorem}[Convolution root]
\label{prop:a3-5}
Assume that $\wt f^{\ast n} \in \cals_+$ for some $n\in\N$, and $\wt f^{\ast n}$ is a.n.i.
If $\wt f^{\ast k},\,k=1\ldots,n-1$ are a.n.i. or if $\wt f\in \call$, then $\wt f\in \cals$. 
\end{theorem} 

\begin{remark}
$(\mathrm{i})$ 
For positive-half densities, the long-tailed property $\call$ plays the most fundamental role for 
deriving various tail properties $($see \cite{Watanabe:2020}$)$. One could see in above that the a.n.i. property could 
play a role of $\call$ in several cases especially for densities on $\R$. However, 
it remains to be seen whether the a.n.i. characteristic has similar properties to that of $\call$  
such as closedness under convolution. An open question is that under what conditions the a.n.i. property is retained.  
\\
$(\mathrm{ii})$
Watanabe and Yamamuro \cite{Watanabe:Yamamuro:2017} have proved that
the class of subexponential densities is neither closed under asymptotic equivalence nor closed under convolution roots. 
Therefore in Lemma \ref{lem:a3-2} and Theorem \ref{prop:a3-5}, we need additional conditions other than the subexponentiality.\\
$(\mathrm{iii})$ Kl\"uppelberg and Villasenor
\cite{Kluppelberg:Villasenor:1991} have negatively solved the convolution closure problem of $\cals$ for 
positive-half densities, and thus we need an additional condition in Proposition \ref{prop:a3-3}, too. 
Interestingly, in their examples, there exist a density $f$ such that $f\in \call$ and $f$ is a.n.i. but $f\not\in \cals$. 
\end{remark}

\section{The compound Poisson case}
\label{sec:cp}
Recall that we always assume a density for the L\'evy measure. 
Let $g$ be a density on $\R$. For $\lambda>0$ define the compound Poisson probability measure by 
\[
 \mu(dx)=e^{-\lambda} \delta_0(dx)+(1-e^{-\lambda})f(x)dx,
\]
where 
\begin{align}
\label{def:cp:acp}
 f(x) = (e^\lambda-1)^{-1} \sum_{n=1}^\infty (\lambda^n /n!)\, g^{\ast n}(x). 
\end{align}
We call $f$ the proper absolutely continuous part of $\mu$. 
\begin{theorem}
\label{thm:compundpoi}
 Suppose that $g$ or equivalently $f$ is bounded. 
The following assertions are equivalent. 
\begin{align*}
(\mathrm{i})\quad & \quad f\in \cals_+\ \text{and}\ f\ \text{is al.d.} \\ 
(\mathrm{ii})\quad & \quad g\in \cals_+\ \text{and}\ g\ \text{is al.d.} \\
(\mathrm{iii})\quad & \quad g\in \call\ \text{and}\ g\ \text{is al.d.}\quad \&\quad \lim_{x\to \infty} f(x)/g(x) =\lambda/(1-e^{-\lambda}). 
\end{align*}
\end{theorem} 
Obviously $\cals_+$ of the theorem can be replaced with $\cals$. 

Different from the one-sided version (Theorem 1.1 of \cite{shimura:watanabe:2022}), where we do not 
need any monotonic-type assumptions, the al.d. condition is added in Theorem \ref{thm:compundpoi}.  
Recall that for a general density $h$, the al.d. condition 
is often supposed to derive $h\in \cals_+ \Rightarrow h \in\cals$. 
Therefore, the al.d. condition in Theorem \ref{thm:compundpoi} is quite natural. 
Moreover, without the al.d. condition 
one can make an example in the two-sided case such that the equivalence of subexponentiality does not hold. 
In Ex.\ref{ex:cp:nonsubexp}, we see that even $g\in \cals_+$ and $g$ is bounded, 
without the al.d. condition it is possible to have $f\notin \cals$.

However, for general $\mu\in {\rm ID}(\R)$ it seems difficult to extend 
the three equivalence relations under the al.d. condition, e.g. it is difficult to derive Lemma \ref{lem:sloc:s}.  
It remains to be seen that to what extent we could weaken the a.n.i. condition in the ${\rm ID}(\R)$ case. 
Different from the distribution, if the density is treated, there exist harder gaps between 
the compound Poisson case and the ${\rm ID}(\R)$ case.  

The proof of Theorem \ref{thm:compundpoi} is given in Subsection \ref{subsec:pr:maintheorem:section4}.
For the proof, we consider the two-sided extension of \cite{shimura:watanabe:2022}, 
and borrow and extend several useful tools in \cite{shimura:watanabe:2022}. The study on the equivalence of above type 
has been initiated in the distribution version \cite[Theorem 3]{Embrechts:Goldie:Veraverbeke:1979}.   
However, direct tracing of the idea in \cite[Theorem 3]{Embrechts:Goldie:Veraverbeke:1979} 
is quite difficult and we need to make new passes. Indeed, even in the positive-half density 
case of \cite{shimura:watanabe:2022}, several new tools have been invented to overcome the difficulty. 

The following auxiliary results specific to a generalized density with a delta function part
are useful, which are interesting and important in themselves.  
Proofs for the results needed for that of Theorem \ref{thm:compundpoi} 
are given before the proof of Theorem \ref{thm:compundpoi} (Subsection \ref{subsec:pr:needed:section4}).
The proofs for not directly related results (Proposition \ref{prop:convroot:gene}, 
Corollary \ref{cor:steutel}) are given in Appendix \ref{append:proofs}.

\begin{proposition}[Factrization]
\label{prop:fact:gene}
 Let $\wt h=\wt g\ast \wt f \in \cals$ such that $\wt g(x)=p_g \delta(x)+(1-p_g)g(x)$ with $p_g\in (2^{-1},1)$. 
Suppose that $\wt h$ is al.d. or $\wt f(x)=p_f \delta(x)+(1-p_f)f(x)$ with $p_f\in (0,1)$. Then, 
$\wt g(x)=o(\wt h(x))$ implies $\wt f\in \cals$, and indeed $\wt h(x)\sim \wt f(x)$. 
\end{proposition}

\begin{corollary}[Factrization, compound Poisson]
\label{prop:fact:cp}
 Let $\wt h=\wt g\ast \wt f \in \cals$ such that $\wt g$ is a compound Poisson.
Suppose that $\wt h$ is al.d. or $\wt f(x)=p \delta(x)+(1-p)f(x)$ with $p\in (0,1)$. 
Then $\wt g=o(\wt h(x))$ implies $\wt f\in \cals$ and indeed $\wt h(x)\sim \wt f(x)$. 
\end{corollary}
\begin{proposition}[Convolution root]
\label{prop:convroot:gene}
Let $\wt f(x)=p \delta(x)+(1-p)f(x)$. 
If $p \in (2^{-1/(n-1)},1)$
then $\wt f^{\ast n} \in \cals$ implies $\wt f\in \cals$. 
\end{proposition}

We make a remark about ``Steutel conjecture'' for a compound Poisson distribution \cite[p.340]{Embrechts:Goldie:Veraverbeke:1979}. 
For the generalized density $\wt f$ of $\mu$ with ch.f. $\wh \mu(z)$, we denote by $\wt f^{\ast \alpha}$ the generalized 
density given by ch.f. $\wh \mu(z)^\alpha$ for any positive real number $\alpha$.  
If $\wt f$ is a compound Poisson with L\'evy density $\lambda g(x)$ with $g$ a density, then so is $\wt f^{\ast \alpha}$
with L\'evy density $\lambda \alpha g(x)$. Due to Theorem \ref{thm:compundpoi} when $g$ or $f$ is bounded, 
then $\wt f\in \cals$ together with the al.d. property implies $\wt f^{\ast \alpha}\in \cals$ for all $\alpha>0$. By applying Proposition \ref{prop:convroot:gene} 
without assuming such conditions for $g$, we have the following.
\begin{corollary}[Steutel cojecture, compound Poisson]
\label{cor:steutel}
 If $\wt f\in \cals$ is the generalized density of a compound Poisson with parameter $\lambda<\log 2$. 
Then $\wt f^{\ast \alpha} \in \cals$ for every rational $\alpha>0$ and $\wt f^{\ast \alpha}(x)/\wt f(x)\to \alpha$ as $x\to\infty$. 
\end{corollary}

We close this section an interesting example. 

\begin{example}
\label{ex:cp:nonsubexp}
 We construct a two-sided compound Poisson $\mu=\mu_-\ast \mu_+$ whose absolutely continuous part of $f$ is 
not subexponential, but the corresponding $\overline f$ of positive-half compound Poisson $\mu_+$ is subexponential 
such that $\overline f$ is not al.d. Here $\mu_-$ is a negative-half compound Poisson.

Suppose that the L\'evy measure density $g_+(x)$ of $\mu_+$ is given by the density of a semistable distribution. 
For $1<x_0 < b$ and $0<2\delta <(x_0-1)\wedge (b-x_0)$, let $\alpha(x)$ be a continuous periodic function 
on $\R$ with period $\log b$ such that $\alpha(\log x)>0$ for $x\in[1,x_0)\cup (x_0,b]$ and 
\begin{align*}
 \alpha(\log x):= \begin{cases}
		 \  0 & \text{for}\ x=x_0\\ 
 \frac{-1}{\log |x-x_0|} & \text{for}\ |x-x_0|<2\delta .
		  \end{cases}
\end{align*}
The L\'evy measure of the semistable distribution is given by $\nu(dx)=x^{-\gamma-1}\alpha(\log x) {\bf 1}_{\{x>0\}} dx$ with $\gamma\in (0,1)$
$($see  \cite[Section 3]{Watanabe:Yamamuro:2017} and \cite[Remarik (iii)]{Watanabe:2020}$)$. 
By the periodic property of the semistable, we have uniformly in $v\in [\delta,2\delta]$ 
\begin{align*}
 \nu\big((b^n(x_0+v),b^n(x_0+v)+1]\big) &\sim  b^{-(\gamma+1)n}(x_0+v)^{-(\gamma+1)} \alpha(\log (x_0+v)),\\
 \nu\big((b^n x_0,b^nx_0+1]\big) & \sim  b^{-(\gamma+1)n}x_0^{-(\gamma+1)}(n\log b)^{-1}. 
\end{align*}
Since $g_+$ is the density of the semistable, Theorem 2 of \cite{Watanabe:2020} 
yields  
\[
 \liminf_{n\to\infty} \frac{g_+(b^n(x_0+v))}{n g_+(b^nx_0)} = \liminf_{n\to\infty} 
\frac{\nu\big((b^n(x_0+v),b^n(x_0+v)+1]\big)}{n\nu\big((b^n x_0,b^nx_0+1]\big)} \ge c_0,
\]
where $c_0$ does not depend on $v\in [\delta,2\delta]$, so that $g_+$ is not al.d., while $g_+ \in \cals_+$ $($see \cite[Remarik (iii)]{Watanabe:2020}$)$.  
Since $\int_0^\infty g_+(x)^2 dx <\infty$, by \cite[Theorem 1.1]{shimura:watanabe:2022} $\overline f (x)\sim c g_+(x)$ and $\overline f\in \cals_+$.

We specify the L\'evy measure $\nu(dx)=g_-(x)dx$ of $\mu_-$. Let $\{n_k\}_{k=1}^\infty,\,n_k\in \N$ be an increasing sequence satisfying 
$\sum_{k=1}^\infty n_k^{-1/2}=1$ such that $n_k\to \infty$ as $k\to\infty$. We set 
$g_-$ as a density on $(-\infty,0)$ such that 
\[
 g_-(x)=\sum_{k=1}^\infty {\bf 1}_{\{ x\in (-2 b^{n_k} \delta,\,-b^{n_k}\delta]\}} b^{-n_k}\delta^{-1} n^{-1/2}_k. 
\]
The absolutely continuous part is given by $\underline{f}(x)= (e-1)^{-1} \sum_{n=1}^\infty g^{\ast n}_-(x)/n! \ge (e-1)^{-1}g_-(x)$. 

Now we consider 
\[
 (1-e^{-2})f(x)=(1-e^{-1})e^{-1}(\underline{f}(x)+\overline{f}(x))+(1-e^{-1})^2 \underline{f}\ast \overline{f}(x)
\]
and see $f \notin \cals$. 
Since $f\in \cals$ $($so that $\wt f\in \cals$$)$ implies $(1-e^{-2})f(x)\sim (1-e^{-1})\overline{f}(x)$ 
by Corollary \ref{prop:fact:cp}, it suffices to show that this does not hold. 
Observe that 
\begin{align*}
 \frac{f(b^n x_0)}{\overline{f}(b^n x_0)} &= \int_{-\infty}^0 \frac{\overline{f}(b^n x_0-y)}{\overline{f}(b^n x_0)} \underline{f}(y) dy \\
 & \ge \int_{-2b^n\delta}^{-b^n \delta} \frac{\overline{f}(b^n x_0-y)}{\overline{f}(b^n x_0)} \underline{f}(y)dy \\
 & \ge c \inf_{v\in [\delta,2\delta]} \frac{\overline{f}(b^n(x_0+v))}{\overline{f}(b^n x_0)} \int_{-2b^n \delta}^{-b^n \delta} \underline{g}(y)dy. 
\end{align*}
Thus, replacing $n$ with $n_k$, we obtain  
\[
 \liminf_{k\to \infty} \frac{f(b^{n_k}x_0)}{\overline{f}(b^{n_k}x_0)} \ge \liminf_{k\to\infty} cn_k (e-1)^{-1}n_k^{-1/2} =\infty,
\]
which implies $(1-e^{-2})f(x)\nsim (1-e^{-1})\overline{f}(x)$. 

As a byproduct, the scaled L\'evy density $g(x)=2^{-1}(g_+(x)+g_-(x))$ of $\mu$ gives an example of 
density such that $g\in \cals_+$ but $g\notin \cals$. We see this by observing behavior of 
\[
 \frac{g^{\ast 2}(x)}{g(x)} \ge \int_{-\infty}^0 \frac{g_+(x-y)g_-(y)}{g_+(x)}dy,\quad x>0. 
\]
Recall that $\overline{f}(x)\sim cg_+(x)$. Then, similarly as before
putting $x=b^{n_k} x_0$ and letting $k\to\infty$, it follows that 
$\liminf_{k\to\infty}g^{\ast 2}(b^{n_k} x_0)/g(b^{n_k} x_0)=\infty$.
\end{example}

\section{Application in the asymptotic theory of statistics}
\label{sec:application:statistics}
We apply our results to the consistency proof of the maximum likelihood estimation (MLE for short) for 
$\mu \in \id$ which is absolutely continuous. 
MLE is the most important estimation in statistics and stands as 
the benchmark for other estimation methods. For simplicity we put $a=b=0$ in $\wh \mu(z)$ of \eqref{def:chf:idr} and assume that 
the spectrally positive part 
$\wh \mu_+(z)$ of \eqref{def:id:spectrally:positive} is absolutely integrable.

Let $f(x;\theta)$ be the density of $\mu$ with $\theta$ a parameter vector 
and $g(x;\theta)$ be a density of the corresponding L\'evy measure $\nu$. 
Let $(X_1,\ldots,X_n)$ be a random sample from $f(x;\theta_0)$ with $\theta_0\in \Theta$ where $\Theta$ is a compact parameter space. 
Define the likelihood function 
\[
 M_n(\theta) = n^{-1} \sum_{i=1}^n \log f(X_i;\theta).
\]
MLE $\hat \theta_n$ maximizes the function $\theta \mapsto M_n(\theta)$. We say that a function $\alpha(x;\theta)$ is 
identifiable if $\alpha(\cdot\, ;\theta) \neq \alpha(\cdot\, ;\theta')$ every $\theta \neq \theta'\in \Theta$, i.e. 
$\alpha(x ;\theta) \stackrel{a.e.}{=} \alpha(x ;\theta')$ does not hold. 
For convenience, we only consider the symmetric or positive-half case, but we can easily generalize the result in the 
non-symmetric two-sided case. We use the function $g_1$ defined in Theorem \ref{theorem:ID:abs}. 
\begin{proposition}
\label{prop:mle}
Let $\mu \in \id$ 
given by \eqref{def:chf:idr} with $a=b=0$ 
such that $\wh \mu_+ (z)$ is absolutely integrable.
Let $g(x;\theta)$ be a symmetric or 
positive-half density of $\nu$. 
Suppose $(\mathrm{i}):$ 
$g(x;\theta)$ is identifiable, 
$\theta \mapsto g(x;\theta)$ is continuous 
in $\theta$ for every $x$, and $\int (\sup_{\theta \in \Theta} |\log g_1(x;\theta)|) g_1(x;\theta_0)dx <\infty$
with $\Theta$ a compact set such that 
$\theta_0 \in \Theta$. Suppose $(\mathrm{ii}):$ $g_1(x;\theta)$ is bounded and a.n.i., and $g_1\in \cals$. 
Then MLE $\hat \theta_n$ satisfies $\hat \theta_n \stackrel{p}{\to} \theta_0$. 
\end{proposition}
The condition  $(\mathrm{i})$ comes from those required for consistency of MLE, while the condition 
 $(\mathrm{ii})$ guarantees $g_1(x)\sim c f(x)$ through Theorem \ref{theorem:ID:abs}. 

\begin{remark}
The proof is done only with the L\'evy density $g$ and the corresponding ch.f., 
and we do not 
touch the genuine density $f$. 
Generally the explicit expression for $f$ of $\mu \in \id$ is unavailable. Nevertheless,  
the estimation methods by $f$ such as MLE are computationally feasible through its ch.f. 
Moreover, there exist many ch.f. based estimators which are comparable with those by $f$ $($cf. \cite{feuerverger:mcDunnough:1981a,yu:2004}$)$.  
For the construction and the theoretical justification $($asymptotics$)$ of these estimators 
the properties of $f$, including the tail asymptotics, are inevitable. 
However, deriving the necessary properties of $f$ from ch.f. are often not so easy, and therefore 
such validity has been proved only for well studied $\mu \in \id$ such as stable laws $($e.g. \cite{dumouchel:1973,matsui:2020}$)$. 
By our results we could extend the scope of these estimators 
to a rather wide subclass of $\id$, which is crucial in applications.

As a next step we are trying to prove the asymptotic normality of MLE for $\mu \in \id$ which is absolutely continuous, 
though it would be much harder than consistency because derivatives of $f$ and $g$ w.r.t. $\theta$ are involved.
\end{remark}

\begin{proof}
 We check the condition of \cite[Theorem 5.7]{vandervaat:2000}. Let $M(\theta)=\E[\log f(X;\theta)]$. 
For consistency $\hat \theta_n \stackrel{p}{\to} \theta_0$, we need two conditions: the deterministic condition
\begin{align}
\label{deterministic:condi}
 \sup_{\theta:|\theta-\theta_0| \ge \vep} M(\theta)<M(\theta_0)
\end{align}
and the stochastic condition 
\begin{align}
\label{stochastic:condi}
 \sup_{\theta\in\Theta} |M_n(\theta)-M(\theta)| \stackrel{p}{\to} 0.
\end{align}
Since $\mu\in {\rm ID}(\R)$ is uniquely defined by the L\'evy measure, the identifiability of $g$ implies that of $f$. Thus \eqref{deterministic:condi}
follows from Lemma 5.35 of \cite{vandervaat:2000}. The condition \eqref{stochastic:condi} is implied by two conditions  \cite[p.46]{vandervaat:2000}: 
$\theta \mapsto \log f(x;\theta)$ are continuous for every $x$ and they are dominated by an integrable envelop function.

For the former condition we use the inversion formula 
and evaluate 
\[
|f(x;\theta)-f(x;\theta')|= \frac{1}{2\pi} 
\Big|\int_{-\infty}^\infty
e^{-izx} (\wh \mu(z;\theta)-\wh \mu(z;\theta'))dz 
\Big|. 
\]
Since $\wh \mu(z;\theta),\,\theta\in \Theta$ is absolutely integrable, we have a dominant integrable function for the integrand. 
Moreover for each $z\in \R$, we have 
\[
|\wh \mu(z;\theta)- \wh \mu(z;\theta')| = |\wh \mu(z;\theta')|\Big|
e^{\int_{-\infty}^\infty(e^{izy}-1-izy {\bf 1}_{\{|z|\le 1\}})(g(y;\theta)-g(y;\theta'))dy }-1 
\Big| \to 0
\] 
as $\theta' \to \theta$. Indeed, since for $z \in \R$
\begin{align*}
& \Big|
\int_{-\infty}^\infty(e^{iz y}-1-izy {\bf 1}_{\{|y|\le 1\}})(g(y;\theta)-g(y;\theta'))dy 
\Big|  \\
& \le \int_{|y|> 1} |e^{iz y}-1||g(y;\theta)-g(y;\theta')|dy + \int_{|y|\le 1} |e^{iz y}-1-izy||g(y;\theta)-g(y;\theta')|dy \\
& \le M+ cz^{2} \int_{ |y| \le 1}y^2|g(y;\theta)-g(y;\theta')| dy < \infty,\quad M>0,
\end{align*}
(cf. \cite[Eq.(8.9)]{sato:1999}) and the dominated convergence works. 
Moreover, in view of the proof $\theta \mapsto f(x;\theta)$ is continuous in $\theta$ uniformly over $x$, and since $\Theta$ is compact it 
is uniformly continuous regardless of values of $x$. 

Before checking the latter condition, we notice that the condition $(\mathrm{ii})$ and the absolute integrability of 
$\wh \mu_+(z)$ imply the condition of Theorem \ref{theorem:ID:abs}, so that  
$g(\pm x;\theta)\sim f(\pm x;\theta)$ follows. 
We proceed to the second condition. 

By symmetry we consider the case $x\ge 0$. 
For fixed $\vep \in(0,1)$ and any $\theta\in \Theta$, there exits $x_\theta\ge 0$ such that 
for all $x>x_\theta$ 
\begin{align}
\label{equive:gf}
 (1-\vep) g(x;\theta) \le f(x;\theta).
\end{align}
For each $\theta$ we take the infimum $x_\theta \ge 0$ without loss of generality. 
If \eqref{equive:gf} holds for all $x>0$ then we put $x_\theta=0$. 
We show that $\ov x =\sup_{\theta\in \Theta} x_\theta <\infty$. If $\ov x=\infty$ then we can choose an 
infinite sequence $(\theta_k,x_{\theta_k})$ such that $x_{\theta_k}\to \infty$ as $k\to\infty$ and 
for all $x>x_{\theta_k}$ \eqref{equive:gf} holds. 
Since $\Theta$ is compact, we can choose a subsequence $(\theta_i,x_{\theta_i})$ such that 
$\lim_{i\to\infty}\theta_i=\theta'\in \Theta$ but $x_{\theta_i}\to\infty$. However, by definition of $x_\theta$, 
this is impossible since $x_{\theta'} <\infty$. Similarly we may let $f(x;\theta)<1$ for all $x>\ov x$. 

Now for all $x> \ov x$ and all $\theta\in \Theta$ we have 
\begin{align*}
|\log f(x;\theta)| &= \log f(x;\theta)^{-1} \\
&\le \log (c_\vep g_1(x;\theta))^{-1} \\
&\le c+\sup_{\theta\in \Theta}|\log g_1(x;\theta)| 
\end{align*}
with $c_\vep>0$ a constant. 
Moreover, $f(x;\theta)$ is continuous in $x$. Thus $\E [\sup_{\theta \in \Theta} |\log f (X;\theta)|] <\infty$
 is implied by the last condition of $(\mathrm{i})$. Now \eqref{stochastic:condi} is proved. 
\end{proof}

\section{Proofs } 
\label{pf:main:ID}
Firstly we give the proofs for results needed for the main results (Theorem \ref{thm:compundpoi} and 
Theorem \ref{theorem:ID}). 
In Subsection \ref{subsec:pr:other:section3} the proofs for auxiliary results in Section \ref{sec:main:results} are given, while in 
Subsection \ref{subsec:pr:needed:section4} those in Section \ref{sec:cp} are given. 
Then we prove Theorem \ref{thm:compundpoi} (Subsection \ref{subsec:pr:maintheorem:section4}) and Theorem \ref{theorem:ID} (Subsection 
\ref{subsec:pr:maintheorem:section5}) in order. 
This is an understandable order.

Throughout this section $c$ denotes a positive constant whose value is not of interest.

\subsection{Proofs for auxiliary results in Section \ref{sec:main:results}}
\label{subsec:pr:other:section3}
\begin{proof}[Proof of {\rm Lemma \ref{lem:equiv:ani:function}}]
 Obviously the a.n.i. property implies \eqref{lem:ani:function}, and we show the converse. 
Observe that 
\[
 1 \le \frac{\sup_{t \ge x} f(t)}{f(x)} = \sup_{t \ge x}\frac{f(t)}{\alpha (x)} \frac{\alpha (x)}{f(x)} 
\le \sup_{t \ge x}\frac{f(t)}{ \alpha (t)} \cdot \frac{\alpha (x)}{f(x)}
\to 1
\]
as $x\to \infty$. Moreover, since $f$ is positive on $[x_0,\infty)$ for some $x_0>0$ and $f(x)\to 0$ as $x\to\infty$, there 
exists $y_x>x$ such that 
\[
 1 \ge \frac{\inf_{x_0\le t \le x} f(t)}{f(x)} \ge  \inf_{x \le t \le y_x}\frac{f(t)}{\alpha (x)} \frac{\alpha (x)}{f(x)} 
\ge \inf_{x \le t \le y_x}\frac{f(t)}{ \alpha (t)}\cdot \frac{\alpha (x)}{f(x)}.
\]
Letting $y_x\to \infty$ and then $x\to \infty$, we obtain the result. Finally it suffices to notice that 
$f\in\call$ implies $f(x)\to 0$ as $x\to\infty$ (see \cite[p.76]{Foss:Korshunov:Zachary:2013}). 
\end{proof}

\begin{proof}[Proof of Lemma \ref{lem:asymp:eqive:ald}]
For given $\vep>0$ there exists $x_0>0$ such that for $x>x_0$ 
\[
 (1-\vep)f(x) \le g(x) \le (1+\vep) f(x)
\]
and we may further take $x_1>x_0$ such that for $y>0$ and $x>x_1$, 
$f(x+y)\le K f(x)$ with some $K>0$. Then for all $y>0$ and $x>x_1$ we have  
\[
 g(x+y) \le (1+\vep) f(x+y) \le (1+\vep) K f(x) \le (1+\vep)/(1-\vep) K g(x). 
\] 
\end{proof}

\begin{proof}[Proof of {\rm Lemma \ref{lem:a3-2}}]
Obviously $\wt g \in \call$.
We take a non-decreasing function $0<\alpha(x)<x/2$ such that $\wt f$ and $\wt g$ are $\alpha$-insensitive (see Lemma \ref{lem:a2-1}). 
Write  
\[
 \frac{\wt g^{\ast 2}(x)}{\wt g(x)} = \Big(2 \int_{-\infty}^{-\alpha(x)}+2 \int_{-\alpha(x)}^{\alpha(x)} + \int_{\alpha(x)}^{x-\alpha(x)} \Big)\frac{\wt g(x-y) \wt g(y)}{\wt g(x)} dy.  
\]
By the dominated convergence 
the second term converges to $2$. If necessary, put $\wt g(x)=p\delta(x) + qg(x)$ and avoid the delta function when 
applying the dominated convergence.  
Moreover, due to \eqref{asmp:lem:a3-2}, the third term is bounded by 
\[
 c \int_{\alpha(x)}^{x-\alpha(x)} \frac{\wt f(x-y)\wt f(y)}{\wt f(x)}dy,
\] 
which is $o(1)$ according to Lemma \ref{lem:a2-1} $(\mathrm{v})$. 

We see that the first term is negligible. For the case $\wt g(-x)=O(\wt f(-x))$ as $x\to\infty$, \eqref{asmp:lem:a3-2}
 together with Lemma \ref{lem:a2-1} $(\mathrm{v})$ implies 
\[
 \int_{-\infty}^{-\alpha(x)} \frac{\wt g(x-y)\wt g(y)}{\wt g(x)} dy \le c \int_{-\infty}^{-\alpha(x)} \frac{\wt f(x-y)\wt f(y)}{\wt f(x)} dy \to 0
\quad {\rm as }\quad x\to\infty.  
\]
If $\wt g$ is al.d. (or $\wt f$ is al.d.), then with some $K>0$ the first term is bounded by 
\begin{align*}
&  \frac{\sup_{t \ge x+\alpha(x)} g(t)}{g(x)} \int_{-\infty}^{-\alpha(x)} \wt g(y)dy \quad \Big( \text{or}\quad  
c \frac{\sup_{t \ge x+\alpha(x)} f(t) }{f(x)} \int_{-\infty}^{-\alpha(x)} \wt f(y)dy \Big) \\
& 
\le K \int_{-\infty}^{-\alpha(x)} \wt g(y)dy \quad \Big(
c K \int_{-\infty}^{-\alpha(x)} \wt f(y)dy \Big)
\to 0, 
\end{align*}
as $x\to\infty$. 
Thus we prove the assertion. 
\end{proof}

\begin{proof}[Proof of {\rm Proposition \ref{prop:a3-3}}]
$(\mathrm{i})$ Owning to Lemma \ref{lem:a3-2}, it suffices to see 
\begin{align}
\label{eq:pf:prop:a3-3}
 \lim_{x\to\infty} \frac{f\ast \wt g(x)}{f(x)}=1. 
\end{align}
Take an insensitive function $\alpha(x)$ for $f$ such that $0<\alpha(x)<x/2$ and $x\to\infty$, and write 
\begin{align*}
 \frac{f\ast \wt g (x)}{f(x)}&= \Big(
\int_{-\infty}^{-\alpha(x)}+\int_{-\alpha(x)}^{\alpha(x)}+\int_{\alpha(x)}^{x-\alpha(x)} +\int_{x-\alpha(x)}^\infty \Big) \frac{f(x-y)}{f(x)}\wt g(y)dy \\
&=:I_1(x)+I_2(x)+I_3(x)+I_4(x). 
\end{align*}
Since $f\in \call$, $\lim_{x\to\infty}I_2(x)=1$ (If it is needed, write $\wt g(x)=p\delta(x)+q g(x)$ and apply the calculation rule for $\delta$). 
We consider the first condition. 
Since $f$ is al.d. and $f\in \call$,  
\[
 I_1(x) \le \frac{\sup_{t\ge x+\alpha(x)}f(t)}{f(x)} \int_{-\infty}^{-\alpha(x)} \wt g(y)dy \to 0\quad {\rm as }\ x\to\infty. 
\]
Moreover, since $\wt g(x)=o(f(x))$ and $f\in \call$, 
in view of $I_1$ and $I_2$ with $\wt g$ replaced by $f$, we have 
\[
 I_4(x) =\int_{-\infty}^{\alpha(x)} \frac{\wt g(x-y)}{f(x-y)}\frac{f(x-y)}{f(x)} f(y)dy \le 
o(1) 
\Big(
\int_{-\alpha(x)}^{\alpha(x)} +\int_{-\infty}^{-\alpha(x)}
\Big) \frac{f(x-y)f(y)}{f(x)}dy \to 0
\]
as $x\to \infty$. Next since $\wt g(x)=o(f(x))$ and $f\in \cals_+$ (cf. Lemma \ref{lem:a2-1} $(\mathrm{v})$)
\begin{align*}
 I_3(x) &= \int_{\alpha(x)}^{x-\alpha(x)} \frac{f(x-y)}{f(x)} \frac{\wt g(y)}{f(y)} f(y) dy \le c \int_{\alpha(x)}^{x-\alpha(x)} \frac{f_+(x-y)}{f_+(x)} f_+(y)dy \to 0 
\end{align*}
as $x\to\infty$. 
Thus we obtain \eqref{eq:pf:prop:a3-3}. Finally, for the second case, it suffices to show that 
$\lim_{x\to\infty}I_1(x)=0$. Since $\alpha(x)\to\infty$ and $\wt g(-x)=O(f(-x))$, 
\[
 I_1(x) \le c \int_{-\infty}^{-\alpha(x)} \frac{f(x-y)f(y)}{f(x)} dy \to 0
\]
follows from Lemma \ref{lem:a2-1} $(\mathrm{v})$. \\
$(\mathrm{ii})$ Let $z\in \R$. 
For sufficiently large $c>0$, 
\begin{align*}
 h(x+z) \ge \int_c^x \wt g(x+z-y)f(y)dy \ge \inf_{y\in[c,x]} f(y) \int_z^{x+z-c} \wt g(y)dy. 
\end{align*}
Then, since $f$ is a.n.i. 
\begin{align*}
\liminf_{x\to \infty} \frac{h(x)}{f(x)} = \lim_{z\to -\infty} \liminf_{x\to \infty} \frac{h(x)}{h(x+z)} \frac{h(x+z)}{\inf_{y\in [c,x]}f(y)} 
\frac{\inf_{y\in [c,x]}f(y)}{f(x)}
\ge \lim_{z\to-\infty} \int_z^\infty \wt g(y)dy =1, 
\end{align*}
which implies 
\begin{align}
 \limsup_{x\to \infty} f(x)/h(x) \le 1. \label{pf:convolution:limsup} 
\end{align}
We will prove 
\begin{align}
 \liminf_{x\to \infty} f(x)/h(x) \ge 1. \label{pf:convolution:liminf} 
\end{align}
Take an insensitive function $0<\alpha(x)<x/2$ for $h$ and write 
\begin{align*}
 1&= \Big(
\int_{-\infty}^{-\alpha(x)}+\int_{-\alpha(x)}^{\alpha(x)}+\int_{\alpha(x)}^{x-\alpha(x)} +\int_{x-\alpha(x)}^\infty \Big) \frac{f(x-y) \wt g(y)}{h(x)}dy \\
& =:I_1(x)+I_2(x)+I_3(x)+I_4(x). 
\end{align*}
Here we may take $\alpha(x)$ to be continuous (see \cite[p.20]{Foss:Korshunov:Zachary:2013}).  
Since $f$ is a.n.i. and $h\in \call$ 
\[
 I_2(x) \le \frac{h(x-\alpha(x))}{h(x)} \frac{f(x-\alpha(x))}{h(x-\alpha(x))}
\frac{\sup_{y \ge x-\alpha(x)}f(y)}{f(x-\alpha(x))}
\int_{-\alpha(x)}^{\alpha(x)} \wt g(y)dy. 
\]
The terms other than $f(x-\alpha(x))/h(x-\alpha(x))$ 
converge to $1$, and we have 
\[
 \liminf_{x\to\infty} I_2(x) \le \liminf_{x\to \infty} f(x)/h(x). 
\]
For $I_1$, we use the a.n.i. property of $f$ and \eqref{pf:convolution:limsup}, i.e. for sufficiently large $x>0$,  
\begin{align*}
 I_1(x) = \frac{f(x+\alpha(x))}{h(x)} \int_{-\infty}^{-\alpha(x)} \frac{f(x-y)}{f(x+\alpha(x))}\wt g(y)dy 
\le (1+\vep) \int_{-\infty}^{-\alpha(x)} \wt g(y)dy \to 0\quad {\rm as}\ x\to\infty.
\end{align*}
Next, we consider 
\[
 I_3(x)=\int_{\alpha(x)}^{x-\alpha(x)} \frac{f(x-y)}{h(x-y)}\frac{\wt g(y)}{h(y)} \frac{h(x-y)h(y)}{h(x)}dy.
\]
Since we have $\wt g(x)=o(h(x))$, \eqref{pf:convolution:limsup} and $h\in \cals_+$, 
\[
 I_3(x) \le c \int_{\alpha(x)}^{x-\alpha(x)} \frac{h_+(x-y)h_+(y)}{h_+(x)}dy \to 0\quad {\rm as }\ x\to\infty 
\]
holds by Lemma \ref{lem:a2-1} $(\mathrm{v})$. 
We study $I_4$. If $h\in \call$ is a.n.i. and then since $\wt g(x)=o(h(x))$, 
\begin{align*}
 I_4(x)&=\int_{-\infty}^{\alpha(x)} \frac{\wt g(x-y)}{h(x-y)} \frac{h(x-y)}{h(x-\alpha(x))}
\frac{h(x-\alpha(x))}{h(x)}  
f(y)dy \\
&\le 
o(1) \frac{\sup_{y\ge x-\alpha(x)}h(y)}{h(x-\alpha(x))} 
\int_{-\infty}^{\alpha(x)} f(y)dy \to 0\quad {\rm as }\ x\to\infty. 
\end{align*}
When $\wt g(x)=o(f(x))$, by \eqref{pf:convolution:limsup} and the a.n.i. property of $f$, 
we have 
\begin{align*}
 I_4(x) &= \int_{-\infty}^{-\alpha(x)} \frac{\wt g(x-y)}{f(x-y)} \frac{f(x-y)}{f(x+\alpha(x))} \frac{f(x+\alpha(x))}{h(x+\alpha(x))}
\frac{h(x+\alpha(x))}{h(x)}  f(y)dy \\
&\quad + \int_{-\alpha(x)}^{\alpha(x)} \frac{\wt g(x-y)}{h(x-y)} \frac{h(x-y)}{h(x)}  
f(y)dy \\
&\le o(1)
\int_{-\infty}^{\alpha(x)} f(y)dy \to 0\quad {\rm as }\ x\to\infty. 
\end{align*}
Thus in view of above results, \eqref{pf:convolution:liminf} follows. 
Now by Lemma \ref{lem:a3-2} together with the a.n.i. property of $f$, we obtain $f\in\cals_+$.
\end{proof}


\subsection{Proofs for results needed for Theorem \ref{thm:compundpoi} in Section \ref{sec:cp}}
\label{subsec:pr:needed:section4}

\begin{proof}[Proof of Proposition \ref{prop:fact:gene}]
With the form $\wt g(x)$, we may write 
\begin{align}
\label{eq:pf:factgene}
 1=p_g \wt f(x)/\wt h(x)+(1-p_g) \wt f \ast g(x)/\wt h(x), 
\end{align}
and put 
\[
\ov C := \limsup_{x\to \infty} \wt f(x)/ \wt h(x)\quad \text{and}\quad \underline{C}:=\liminf_{x\to\infty} \wt f(x)/\wt h(x), 
\]
which are well-defined since $\wt h(x)\ge p_g \wt f(x)$ and the Dirac delta parts disappear for $x>0$.

Assume the first condition. 
Take an insensitive function $\alpha$ for $\wt h$ and consider 
\begin{align*}
 \frac{\wt f \ast g(x)}{\wt h(x)} &= \Big(
\int_{-\infty}^{-\alpha(x)}+ \int_{-\alpha(x)}^{\alpha(x)} +\int_{\alpha(x)}^\infty 
\Big)  \frac{\wt f(x-y) g(y)}{\wt h(x)}dy \\
& =: I_1(x)+I_2(x)+I_3(x). 
\end{align*}
By Fatou's lemma and $\wt h\in\call$ 
\begin{align}
 \liminf_{x\to\infty} I_2(x) &= \liminf_{x\to\infty} \int_{-\alpha(x)}^{\alpha(x)} 
\frac{\wt f(x-y)}{\wt h(x-y)} \frac{\wt h(x-y)}{\wt h(x)} g(y)dy \label{ineq:pf:I2i} \\
& \ge \int_{-\infty}^\infty \liminf_{x \to\infty} \frac{\wt f(x-y)}{\wt h(x-y)}
{\bf 1}_{\{y\in[-\alpha(x),\alpha(x)]\}} g(y)dy \ge \underline{C}, \nonumber
\end{align}
where we may take a continuous $\alpha(x)$ if needed. 
Again by Fatou's lemma 
\begin{align}
 \limsup_{x\to \infty} I_2(x) \le \int_{-\infty}^\infty \limsup_{x\to\infty} 
\frac{\wt f(x-y)}{\wt h(x-y)} {\bf 1}_{\{y\in [-\alpha(x),\alpha(x)]\}} g(y)dy \le \ov C.
\label{ineq:pf:I2s}
\end{align}
Since $\wt h(x) \ge p_g \wt f(x)$ by \eqref{eq:pf:factgene} and $\wt g(x)=o(\wt h(x))$ we have 
\begin{align}
 \limsup_{x\to\infty} I_3(x) &= \limsup_{x\to\infty} \int_{-\infty}^{x-\alpha(x)}
\frac{g (x-y)}{\wt h(x-y)}\frac{\wt h(x-y)}{\wt h(x)} p_g^{-1} \wt h(y) dy  \label{ineq:pf:I3} \\
& \le \limsup_{x \to \infty} \frac{g(x)}{\wt h(x)}  c \limsup_{x\to \infty} \frac{\wt h^{\ast 2}(x)}{\wt h(x)} =0.  
\nonumber
\end{align}
By the al.d. property of $\wt h$, it follows that 
\begin{align}
\limsup_{x\to\infty} I_1(x) &= \limsup_{x\to\infty} \int_{-\infty}^{-\alpha(x)}
 \frac{\wt f (x-y)}{\wt h(x-y)} \frac{\wt h(x-y)}{\wt h(x)} g(y)dy \label{ineq:pf:I4} \\
& \le c \limsup_{x\to\infty} \int_{-\infty}^{-\alpha(x)} g(y)dy =0. \nonumber
\end{align}
Collecting 
\eqref{ineq:pf:I2i}-\eqref{ineq:pf:I4}, we obtain 
\begin{align}
 \liminf_{x\to\infty} \frac{\wt f\ast g (x)}{\wt h(x)} \ge \underline{C}\quad \text{and}\quad \limsup_{x\to\infty} 
\frac{\wt f \ast g (x)}{\wt h(x)} \le \ov{C}. 
\end{align}
Now taking $\liminf_{x\to\infty}$ and $\limsup_{x \to \infty}$ on both sides of \eqref{eq:pf:factgene}, we have 
\begin{align}
\label{ineq:pf:factgene}
 p_g \ov C +(1-p_g) \underline{C}\ \le  1\  \le\ p_g \underline{C} +(1-p_g) \ov C 
\end{align}
and thus 
\[
 0\le (1-2 p_g)(\ov C-\underline{C}).
\]
The assumption $p_g \in (2^{-1},1)$ implies $\ov C=\underline{C}$. Moreover from \eqref{ineq:pf:factgene}, $\ov C=\underline{C}=1.$ 
Then noticing $p_g \wt f(x) \le \wt h(x)$ for all $x\in\R$, we have $\wt f\in \cals$ from Lemma \ref{lem:a3-2}.   

Next assume the second condition. In view of the expression 
\begin{align}
 \wt h(x) =p_g p_f \delta(x) +(1-p_f) p_g f(x)+(1-p_g)p_f g(x) +(1-p_g)(1-p_f) f\ast g(x),
\end{align}
we notice that $\wt h(x) \ge c g(x)$ and $\wt h(x) \ge c f(x)$ hold for all $x\in \R$. From the first part 
proof, it suffices to show that $\limsup_{x\to\infty} I_1(x)=0$ where the al.d. property of $\wt h$ is used. 
However, by above inequalities 
\begin{align*}
 \limsup_{x\to\infty} I_1(x) &= \limsup_{x\to\infty} \int_{-\infty}^{-\alpha(x)} 
\frac{\wt f(x-y)}{\wt h(x-y)} \frac{g (y)}{\wt h(y)} 
\frac{\wt h(x-y) \wt h(y)}{\wt h(x)} dy \\
& \le c  \limsup_{x\to\infty} 
\int_{-\infty}^{-\alpha(x)} \frac{\wt h(x-y) \wt h(y)}{\wt h(x)} dy =0. 
\end{align*}
\end{proof}

\begin{proof}[Proof of Corollary \ref{prop:fact:cp}]
 We assume non-degeneracy for $\wt g$, since otherwise the proof is obvious. 
Let $\lambda$ be the Poisson parameter of $\wt g$. 
If $\lambda < \log 2$, then since $e^{-\lambda}>2^{-1}$ 
the result is immediate from Proposition \ref{prop:fact:gene}.
If $\lambda \ge \log 2$, we take an integer $n$ such that $\lambda/n <\log 2$, and 
define a compound Poisson density 
\[
 \wt g_{1/n}(x)= e^{-\lambda/n}\delta(x)+(1-e^{-\lambda/n})g_{1/n}(x)
\]
with $g_{1/n}$ be the proper absolutely continuous part such that $\wt g=\wt g_{1/n}^{\ast n}$. 
Notice that since 
\[
 \wt g(x) = \big(e^{-\lambda/n}\delta + (1-e^{-\lambda/n})g_{1/n} \big)^{\ast n}(x) \ge n e^{-\lambda(n-1)/n}(1-e^{-\lambda/n})g_{1/n}(x) 
\]
and thus $\wt g(x)=o(\wt h(x))$ implies $\wt g_{1/n}(x)=o(\wt h(x))$. Moreover, the coefficient $e^{-\lambda k/n}p$ of 
the delta part $\delta$ in $\wt g_{1/n}^{\ast k}\ast \wt f$ satisfies $e^{-\lambda k/n}p\in (0,1),\,k=1,\ldots,n$. 
Now we apply Proposition \ref{prop:fact:gene} to 
$\wt g_{1/n}\ast (\wt g_{1/n}^{\ast (n-1)}\ast \wt f)$ with $\wt g_{1/n}$ be the negligible part, 
and obtain that $\wt g_{1/n}^{\ast (n-1)}\ast \wt f \in \cals$ and $\wt h(x)\sim \wt g_{1/n}^{\ast (n-1)}\ast \wt f(x)$.
Here, if $\wt h$ is al.d., then by Lemma \ref{lem:asymp:eqive:ald}, $\wt g_{1/n}^{\ast (n-1)}\ast \wt f(x)$ is also al.d. 
We iterate this step until we reach $\wt f\in \cals$ and $\wt h(x)\sim \wt f(x)$. 
\end{proof}

\subsection{Proof of Theorem \ref{thm:compundpoi}}
\label{subsec:pr:maintheorem:section4}
Firstly we state an auxiliary lemma and then go to the proof of Theorem \ref{thm:compundpoi}. 
\begin{lemma}
\label{lem:cp:lower}
Let $\mu$ be a compound Poisson with ch.f.  
\begin{align*}
\wh \mu (z) & = \exp\Big(
\lambda\int_{-\infty}^{c_1} (e^{izy}-1)g(y)dy
\Big), 
\end{align*}
where $g$ is bounded. Then for any $c_1>0$ there exists $\gamma >0$
such that the absolutely continuous part $f$ of $\mu$ satisfies $f(x)=o(e^{-\gamma x})$. 
\end{lemma}

\begin{proof}
We decompose $\wh \mu(z)$ into 
\[
 \wh \mu(z)= \exp \Big\{
\lambda \Big(
\int_{-\infty}^0 + \int_{0}^{c_1}\Big) (e^{izy}-1)g(y)dy
\Big\}=: \wh \mu_{1}(z) \wh \mu_{2}(z). 
\]
First we consider the compound Poisson $\mu_2$ and let $\Lambda_2=G(c_1)-G(0)<\infty$. 
We write the proper absolutely continuous part as 
\[
 f_2(x)=(e^{\lambda \Lambda_2}-1)^{-1} \sum_{n=1}^\infty \big( (\lambda \Lambda_2)^n/n! \big)\, g_2^{\ast n}(x), 
\]
where $g_2(x)=\Lambda_2^{-1} g(x){\bf 1}_{\{0\le x \le c_1\}}$. Since $g$ is bounded, $f_2(x)$ is bounded as well. 
Recall that for any $\gamma>0,\,\int_{[0,\infty)} e^{\gamma x} \mu_2(dx)<\infty$ (\cite[Theorem 25.3]{sato:1999}). So from e.g. \cite[Ex. 33.15]{sato:1999} 
(cf. \cite[Lemma 7]{Watanabe:2020} and \cite[Theorem 3.9]{Kuprianou:2006}), we can define the exponential tilt $(\mu_{2})_\gamma$ of $\mu_{2}$ as 
\[
 (\mu_{2})_\gamma (dx) = \frac{e^{\gamma x}}{\int_{[0,\infty)} e^{\gamma x}\mu_{2}(dx)} \mu_{2}(dx). 
\]
Then $(\mu_2)_\gamma$ is again the compound Poisson with the proper absolutely continuous part 
\[
 f_e(x)= (e^{\lambda \Lambda_e}-1)^{-1} \sum_{n=1}^\infty ((\lambda \Lambda_{e})^n /n!)\, g_{e}^{\ast n}(x)
\]
with $\Lambda_e=\int_0^{c_1} e^{\gamma x}g (x)dx$ and $g_e(x)=(\Lambda_2/\Lambda_e) e^{\gamma x} g_2(x)$. 
For any $x>0$, the right-hand side is well defined. 
Since the support of $g_{e}^{\ast n}$ is included in the interval $[0,nc_1]$, we have 
\[
 \lim_{x\to \infty} e^{\gamma x}f_2(x) = \lim_{x\to\infty} cf_e(x)= c\lim_{x\to\infty} \sum_{nc_1\ge x}^\infty
((\lambda \Lambda_e)^n /n!)\,  g_{e}^{\ast n}(x)=0. 
\]
Now since $\mu_1$ is a compound Poisson with non-positive support, 
it suffices to check that for the absolutely continuous part $f_1$ of $\mu_1$,  
\begin{align}
\label{eq:convolution:fa1fa2}
 e^{\gamma x}f_1\ast f_2(x) &= \int_{-\infty}^\infty e^{\gamma (x-y)}f_2(x-y)e^{\gamma y} f_1(y)dy \\
&\le o(1) \int_{-\infty}^0 e^{\gamma y}f_1(y)dy. \nonumber
\end{align} 
Thus we may take some $\gamma>0$ such that $f(x)=o(e^{-\gamma x})$. 
\end{proof}

\begin{proof}[Proof of Theorem \ref{thm:compundpoi}]
Denote   
the Laplace transform of $f$ by 
$L_{f}(z)=\int_0^\infty e^{z x}f(x)dx$. \\

\noindent
{\bf $(\mathrm{i})$ implies $(\mathrm{ii})$ and $(\mathrm{iii})$} \\
Since $g$ is a proper density, we may choose $c_1>0$ such that $\ov G(c_1)=\Lambda_1<\log 2/\lambda$ and define another 
compound Poisson by 
\[
\wh \mu_1 (z) = \exp\Big(
\lambda \Lambda_1 \int_{c_1}^{\infty} (e^{izy}-1)g_1(y)dy
\Big),\quad  g_1(x)=g(x)/\Lambda_1, 
\]
so that 
\[
 \mu=\mu_1\ast \mu_2,\quad i.e.\quad \wh \mu(z)=\wh \mu_1(z)\wh \mu_2(z),
\]
where 
\[
\wh \mu_2 (z) = \exp\Big(
\lambda \Lambda_2 \int_{-\infty}^{c_1} (e^{izy}-1)g(y)/\Lambda_2 dy
\Big),\quad  \Lambda_2=G(c_1). 
\]
Let $\wt f_1$ and $\wt f_2$ be generalized densities of $\mu_1$ and $\mu_2$ respectively. 
By Lemma \ref{lem:cp:lower} $\wt f_2(x)=o(e^{-\gamma x})$ with some $\gamma>0$, so that 
$\wt f_2(x)=o(\wt f(x))$.
Now apply Corollary \ref{prop:fact:cp} with $(\wt h,\wt g, \wt f)$ there be $(\wt f,\wt f_2,\wt f_1)$ here and 
obtain $\wt f_1\in\cals(=\cals_+)$ and $\wt f_1(x)\sim \wt f(x)$. 

For the proper densities $f_1$ of $\wt f_1$, 
since $\wt f_1 \in \cals \Leftrightarrow f_1\in \cals$, 
we will see that 
$f_1\in \cals_+$ implies $g_1\in \cals_+$. This part is totally due to \cite[Theorem 1]{shimura:watanabe:2022} and for 
consistency we give a proof. 
Write 
\begin{align}
\label{def:tranc:cp}
 f_1(x) = (e^{\lambda \Lambda_1}-1)^{-1} \sum_{n=1}^{\infty}((\lambda \Lambda_1)^n/n!) g_1^{\ast n}(x) 
\end{align}
whose Laplace transform is 
\[
 L_{f_1}(z)=(e^{\lambda \Lambda_1 L_{g_1}(z)}-1)/(e^{\lambda \Lambda_1}-1),
\]
so 
\[
 \lambda \Lambda_1 L_{g_1}(z)=\log \big( 1-(1-e^{\lambda \Lambda_1})L_{f_1}(z) \big). 
\]
Since $e^{\lambda \Lambda_1}-1 <1$, we have 
\[
 \lambda \Lambda_1 L_{g_1}(z)= -\sum_{n=1}^\infty n^{-1}(1-e^{\lambda \Lambda_1})^n L^n_{f_1}(z) 
\]
and thus 
\begin{align}
\label{def:inverse:chf}
 \lambda_1 \Lambda_1 g_1(x) \stackrel{a.e.}{=} -\sum_{n=1}^\infty n^{-1} (1-e^{\lambda \Lambda_1})^n f_1^{\ast n}(x)=:\lambda_1 \Lambda_1
\breve g_1(x), 
\end{align}
where $\stackrel{a.e.}{=}$ implies that the equality holds $a.e.\,x\in \R.$

 We derive $\breve{g}_1 \in \cals_+$ and the 
tail equivalence between $f_1$ and $\breve{g}_1$. Then using \eqref{def:tranc:cp} we prove $g_1 \in \cals_+$. 
Take a sufficiently small $\vep >0$ such that $(e^{\lambda} -1)(1+\vep)<1$. 
By Lemma \ref{lem:kestenbound} 
there exists $C_\vep$ such that $f^{\ast n}_1(x) \le C_\vep (1+\vep)^n f_1(x)$ for $x$ sufficiently large. 
Applying the dominated convergence in \eqref{def:inverse:chf} we obtain 
\begin{align}
\label{tail:equiv:gf:symmetric}
 \lim_{x\to\infty} \breve g_1(x)/f_1(x) =(1-e^{-\lambda \Lambda_1})/(\lambda\Lambda_1), 
\end{align}
so that by Lemma \ref{lem:a2-1}
$(\mathrm{vi})$ 
$\breve g_1 \in\cals_+$. 
Now write \eqref{def:tranc:cp} as 
\begin{align*}
 (e^{\lambda \Lambda_1}-1)f_1(x)/\breve g_1(x)-\lambda \Lambda_1 g_1(x)/\breve g_1(x) &= 
\sum_{n=2}^\infty ((\lambda \Lambda_1)^n/n!)\, g^{\ast n}_1(x)/\breve g_1(x) \\
&= \sum_{n=2}^\infty ((\lambda \Lambda_1)^n/n!)\, \breve g^{\ast n}_1(x)/\breve g_1(x)
\end{align*}
and let $x\to\infty$. Again by the dominated convergence we obtain by \eqref{tail:equiv:gf:symmetric} that $\breve g_1(x) \sim g_1(x)$.  
Thus we prove $g_1 \in\cals_+$ and $g_1(x)\sim (1-e^{-\lambda \Lambda_1})/(\lambda\Lambda_1) f_1(x)$.
Now we see 
\begin{align}
\label{eq:ratio:f1g1}
 \lim_{x\to\infty} \frac{f(x)}{g(x)}=\lim_{x\to\infty} \frac{f(x)}{f_1(x)} \frac{f_1(x)}{g_1(x)} \frac{g_1(x)}{g(x)}=\lambda/(1-e^{-\lambda}). 
\end{align}
In addition, due to Lemma \ref{lem:asymp:eqive:ald}, the al.d. property of $g$ is implied by that of $f$ and $(\mathrm{iii})$ follows. 
Moreover, by Lemma \ref{lem:a3-2}, $(\mathrm{ii})$: $g\in \cals$ is immediate. \\


\noindent
{\bf $(\mathrm{iii})$ implies  $(\mathrm{ii})$} \\
We write 
\[
 g^{\ast 2}(x) = (2/\lambda^2 e^{\lambda}) \big\{
(1-e^{-\lambda}) f(x) -e^{-\lambda} \sum_{n\neq 2}^\infty (\lambda^n/n!)\, g^{\ast n}(x)
\big\}.
\] 
Dividing this by $g$ and taking $\limsup$ on both sides, we have by Fatou's lemma 
\begin{align*}
 \limsup_{x\to \infty} \frac{g^{\ast 2}(x)}{g(x)} 
& \le (2/\lambda^2) e^\lambda \Big(
\lambda -e^{-\lambda} \sum_{n\neq 2}^\infty \liminf_{x\to\infty} \frac{\lambda^n}{n!} \frac{g^{\ast n}(x)}{g(x)}
\Big)=2, 
\end{align*}
where we use Lemma \ref{lem:a2-1} $(\mathrm{iii})$.
Now
Lemma \ref{lem:a2-1} $(\mathrm{iii})$ with $n=2$ 
implies 
$g\in \cals$. \\
 
\noindent
{\bf $(\mathrm{ii})$ implies $(\mathrm{iii})$ and $(\mathrm{i})$} \\
Lemma \ref{lem:kestenbound} holds with $g$ by al.d. property.
Thus applying $g^{\ast n}(x)/g(x) \to n$ (\cite[Theorem 1.1.]{Finkelshtein:Tkachov:2018} for the two-sided case) and 
the dominated convergence to \eqref{def:cp:acp}, we show that $(\mathrm{ii})$ implies $(\mathrm{iii})$. 
Then $(\mathrm{ii})$ and $(\mathrm{iii})$ together with the al.d. property of $g$ 
yield $f\in \cals$ by Lemma \ref{lem:a3-2}. 
\end{proof}

\subsection{Proof of Theorem \ref{theorem:ID}}
\label{subsec:pr:maintheorem:section5}
The proof of the part [$(\mathrm{i})$ implies $(\mathrm{ii})$ and $(\mathrm{iii})$ under the condition that $g$ is a.n.i.]
 is the hardest part to come up with the proof idea and we need a few extended notion of $\call$ and $\cals$ and 
auxiliary results before. 

\begin{definition}
$(\mathrm{i})$ Let $\Delta:=(0,c]$ with $c>0$. $F$ belongs to the class $\call_\Delta$ if $F(x+\Delta):=F(x+c)-F(x)\in \call$. 
Moreover, $F$ belongs to the class $\cals_\Delta$ if $F\in \call_\Delta$ and $F^{\ast 2}(x+\Delta)\sim 2F(x+\Delta)$. \\\
$(\mathrm{ii})$ $F$ belongs to the class $\call_{loc}$ if $F\in \call_\Delta$ for all $\Delta=(0,c]$ with $c>0$, and moreover, 
$F$ belongs to the class $\cals_{loc}$ if $F\in \cals_\Delta$ for all $\Delta$. 
\end{definition}
Obviously $\cals_\Delta \subset \cals_{loc}$. The following connects $\cals$ and $\cals_{loc}$, which is a partial result 
in Matsui and Watanabe\footnote{''A note on subexponentiality'', unpublished manuscript}. 

\begin{lemma}
\label{lem:sloc:s}
 Let $F$ be a distribution on $\R_+$ and $f$ be its density. If $f$ is a.n.i., then 
$F\in \cals_{loc} \Leftrightarrow f\in \cals$. 
\end{lemma}
\begin{proof}[Proof of Theorem \ref{lem:sloc:s}]
First, we see $f\in \call$, so that $f^{\ast 2}\in \call$. Take 
a sufficiently large $x_0>0$ such that $\inf_{t \in[x_0,x]} f(t) \sim f(x)$. 
For any $y\in \R$
\begin{align*}
 \frac{f(x+y)}{f(x)} &= \frac{\sup_{t\ge x+y}f (t)}{\inf_{s\in [x_0,x]} f(s)}
\frac{f(x+y)}{\sup_{t\ge x+y}f(t)} \frac{\inf_{s\in [x_0,x]}f(s)}{f(x)} \\
& \ge \frac{F(x+y+\Delta)}{F(x-c+\Delta)}
\frac{f(x+y)}{\sup_{t\ge x+y}f(t)} \frac{\inf_{s\in [x_0,x]}f(s)}{f(x)} 
\end{align*}
and moreover, 
\begin{align*}
 \frac{f(x+y)}{f(x)} &= \frac{\inf_{t \in [x_0, x+y]} f (t)}{\sup_{s\ge x} f(s)}
\frac{f(x+y)}{ \inf_{t \in [x_0, x+y]} f(t)} \frac{\sup_{s\ge x}f(s)}{f(x)} \\
& \le \frac{F(x+y-c+\Delta)}{F(x+\Delta)}
\frac{f(x+y)}{ \inf_{t \in [x_0, x+y]}f(t)}  \frac{\sup_{s\ge x}f(s)}{f(x)}.
\end{align*}
Now since $F\in \call_\Delta$ and $f$ is a.n.i., by taking $\lim_{x\to \infty}$ in both inequality we have 
\[
 1 =\lim_{x\to\infty}\frac{F(x+y+\Delta)}{F(x-c+\Delta)} \le \lim_{x\to\infty}
\frac{f(x+y)}{f(x)} \le \lim_{x\to\infty} \frac{F(x+y-c+\Delta)}{F(x+\Delta)}=1 
\]
and obtain $f\in \call$. Then the uniform convergence property of $\call$ yields 
\[
 F(x+\Delta) \sim c f(x)\quad \text{and}\quad F^{\ast 2}(x+\Delta)\sim cf^{\ast 2}(x)\quad \text{for all}\quad c>0.
\]
Hence $F\in \cals_{loc} \Leftrightarrow f\in \cals$. 
\end{proof}
\begin{lemma}
\label{lem:factorization:id}
 Let $\mu \in \id$ with $\nu(dx)=g(x)dx$ such that $\nu(\R)=\infty$ and denote a density of $\mu$ by $f$. Let $\mu_1$
 be a compound Poisson with L\'evy measure $\nu(dx)={\bf 1}_{\{x\le -1\}}g(x)dx$ and $\gamma=0$, and let $\mu_2 \in \id$ 
such that $\mu=\mu_1\ast \mu_2$, i.e. $\wh \mu_2(z)$ is given by $\wh \mu(z)$ of \eqref{def:chf:idr}
with $\nu(dx)={\bf 1}_{\{x>-1\}}g(x)dx$. Denote a density of $\mu_2$ by $f_2$. Then, if 
$f\in \cals$ and $f$ is al.d., then $f_2\in \cals$ and $f(x)\sim f_2(x)$. 
\end{lemma}
\begin{proof}
Putting $\wt h=f,\,\wt g=f_1$ and $\wt f=f_2$ in Corollary \ref{prop:fact:cp}, immediately we obtain the result.   
\end{proof}

For the proof of parts [$(\mathrm{iii}) \Rightarrow (\mathrm{ii})$ and $(\mathrm{ii}) \Rightarrow (\mathrm{i})$ and $(\mathrm{iii})$],
we need the following lemma.
\begin{lemma}
\label{lem:pf:main:gamma:moment}
For $c_1>1$, define $\mu_r\in\id$ with ch.f. 
\begin{align}
\label{def:mua}
 \wh \mu_r(z) = \exp \Big\{
\int_{-\infty}^{c_1} (e^{izy}-1-izy {\bf 1}_{\{|y|\le 1\}})g(y)dy +iaz -\frac{1}{2} b^2 z^2
\Big\}.
\end{align}
Under the condition of Theorem \ref{theorem:ID}, the density $f_r$ of $\mu_r$ satisfies 
\begin{align}
\label{fa:fr:exponential}
 \lim_{x\to\infty}e^{\gamma x}f_r(x)=0.
\end{align}
\end{lemma}
\begin{proof}
 We study the case $a=b=0$ in $\mu_r$ and then generalize the result.
We decompose $\wh \mu_r(z)$ into 
\[
 \wh \mu_r(z)= \exp \Big\{
\Big(
\int_{-\infty}^{-1} +\int_{-1}^1 + \int_{1}^{c_1}\Big) (e^{izy}-1-izy {\bf 1}_{\{|y|\le 1\}})g(y)dy
\Big\}=: \wh \mu_{r_1}(z) \wh \mu_{r_2}(z)\wh \mu_{r_3}(z). 
\]
Consider the proper absolutely continuous part $f_{r_3}$ of $\mu_{r_3}$. Since $g$ is bounded on $[1,c_1]$, by  
exactly the same logic as for $f_2$ in the proof of Lemma \ref{lem:cp:lower}, we have $f_{r_3}(x)=o(e^{-\gamma x})$. 
Since $\mu_{r_1}$ is a compound Poisson with a non-positive support, again by the same reasoning as in the proof of Lemma \ref{lem:cp:lower}
(cf. \eqref{eq:convolution:fa1fa2}), 
the absolutely continuous part $f_{r_{13}}$ of $\mu_{r_1}\ast \mu_{r_3}$ satisfies $f_{r_{13}}(x)=o(e^{-\gamma x})$.

Now consider the convolution of $\mu_{r_2}(dx)=f_0(x)dx$ and 
\[
 \mu_{r_1}\ast \mu_{r_3}(dx) := e^{-c_{13}} \delta_0(dx) +(1-e^{-c_{13}})f_{13}(x)dx,
\]
which yields 
\begin{align*}
 e^{\gamma x}f_r(x) = e^{\gamma x} e^{-c_{13}} f_0(x) &+ (1-e^{-c_{13}}) \int_{-\infty}^{x/2} e^{\gamma(x-y)} f_0(x-y)e^{\gamma y}f_{13}(y)dy \\
&+(1-e^{-c_{13}}) \int_{-\infty}^{x/2} e^{\gamma(x-y)} f_{13}(x-y)e^{\gamma y}f_0(y)dy. 
\end{align*}
Recall that both $\mu_{r_2}$ and $\mu_{r_1}\ast \mu_{r_3}$ have $e^{\gamma x}$ moment \cite[Theorem 25.3]{sato:1999}, and so do $f_{13}$ and $f_0$. 
Moreover, both $e^{\gamma x}f_0(x)$ and $e^{\gamma x} f_{13}(x)$ converge to $0$ as $x\to\infty$. Thus by the dominated convergence we have 
$\lim_{x\to\infty} e^{\gamma x} f_r(x)=0$.

Finally let $f_G$ the density of Gaussian part plus the shift, and then 
\[
 e^{\gamma x} f_r\ast f_G (x) = \int_{-\infty}^{x/2}\big\{
e^{\gamma(x-y)}f_r(x-y) e^{\gamma y} f_G(y) +e^{\gamma(x-y)}f_G(x-y)e^{\gamma y}f_r(y)
\big\} dy <\infty 
\]
follows by exactly the same way as before.
\end{proof}

\begin{proof}[Proof of Theorem \ref{theorem:ID}]
{\bf $(\mathrm{i})$ implies $(\mathrm{ii})$ and $(\mathrm{iii})$ under the condition that $g$ is a.n.i.} \\
We decompose $\mu$ as in Lemma \ref{lem:factorization:id}: $\mu=\mu_1\ast \mu_2$. 
Then $(\mathrm{i})$ implies that $f_2\in \cals$ and $f(x)\sim f_2(x)$, so that $f_2$ is al.d. by Lemma \ref{lem:asymp:eqive:ald}. 
Notice that $f_2 \in \cals$ implies $\mu_2 \in \cals $ and $\mu_2$ has any $e^{-\gamma x}$ moment with $\gamma>0$ 
\cite[Theorem 25.3]{sato:1999}.
Now \cite[Corollary 3.1]{Watanabe:Yamamuro:2010} yields $\nu_1 \in \cals_{loc}$. Due to Lemma \ref{lem:sloc:s}, the a.n.i. property of 
$g$ implies $g_1\in \cals_+$. 
\\
\noindent
{\bf $(\mathrm{iii})$ implies $(\mathrm{ii})$ under the condition that $g$ is al.d.} \\
For $c_1>1$ define $g_u(x)=g(x){\bf 1}_{\{x\ge c_1\}}/\Lambda_u$ where $\Lambda_u=\ov G(c_1)$ and consider the compound Poisson $\mu_u$ with ch.f. 
\[
 \wh \mu_u(z)=\exp \Big\{
\Lambda_u \int_{c_1}^\infty (e^{izy}-1)g_u(y)dy
\Big\},
\] 
and the proper absolutely continuous part  
\[
 f_u(x)=(e^{\Lambda_u}-1)^{-1} \sum_{n=1}^\infty (\Lambda_u^n/n!) g_u^{\ast n}(x). 
\]
By Fubini the density $f$ of $\mu=\mu_r \ast \mu_u$ has an expression 
\[
 f(x)= e^{-\Lambda_u}f_r(x) + e^{-\Lambda_u} \sum_{n=1}^\infty (\Lambda_u^n/n!) g_u^{\ast n}\ast f_r(x),  
\]
where $f_r$ is the density of $\mu_r$. 
We write 
\begin{align}
\label{eq:inv:poisson}
 \frac{g_u^{\ast 2}\ast f_r(x)}{g_u(x)}= e^{\Lambda_u} \frac{2}{\Lambda_u^2} \frac{f(x)}{g_u(x)}-\frac{2}{\Lambda_u^2}\frac{f_r(x)}{g_u(x)}-
\frac{2}{\Lambda_u^2 g_u(x)} \sum_{n\neq 2}^\infty \frac{\Lambda_u^n}{n!}g_u^{\ast n}\ast f_r(x)
\end{align}
and observe the limit behavior when $x\to\infty$. 
Since $g_u(x)\sim g(x)/\Lambda_u$, due to the second condition of $(\mathrm{iii})$ together with $g(x)\sim g_1(x)\nu((1,\infty))$, 
\[
 \lim_{x\to\infty} e^{\Lambda_u} \frac{2}{\Lambda_2^2}\cdot \frac{f(x)}{g_u(x)}
= \lim_{x\to\infty} e^{\Lambda_u} \frac{2}{\Lambda_2^2}\cdot \frac{f(x)}{g_1(x)} \frac{g_1(x)}{g(x)} \frac{g(x)}{g_u(x)} 
=e^{\Lambda_u}\frac{2}{\Lambda_u}. 
\]
Since $f_r(x)=o(e^{-\gamma x})$ for some $\gamma>0$ by Lemma \ref{lem:pf:main:gamma:moment} and $g_u\in \call$, the second quantity vanishes. 
Moreover, by Lemma \ref{lem:a2-1} $(\mathrm{iii})$ 
\begin{align}
\label{liminf:gu:fr}
 \liminf_{x\to\infty} \frac{g_u^{\ast n}\ast f_r(x)}{g_u(x)} \ge \liminf_{x\to\infty} \frac{g_u^{\ast n}(x)}{g_u(x)} \liminf_{x\to\infty} 
\frac{g_u^{\ast n}\ast f_r(x)}{g_u^{\ast n}(x)} \ge n.   
\end{align}
Thus, Fatou's lemma yields   
\begin{align*}
 \liminf_{x\to\infty} \sum_{n\neq 2}^\infty \frac{\Lambda_u^n}{n!} \frac{g_u^{\ast n}\ast f_r(x)}{g_u(x)} 
\ge  \sum_{n\neq 2}^\infty \frac{\Lambda_u^n}{n!} \liminf_{x\to\infty} \frac{g_u^{\ast n}\ast f_r(x)}{g_u(x)}
\ge \Lambda_u(e^{\Lambda_u}-\Lambda_u).  
\end{align*}
Considering $\limsup$ in  
\eqref{eq:inv:poisson} with above results including \eqref{liminf:gu:fr} with $n=2$, we have  
\[
 \lim_{x\to\infty} \frac{g_u^{\ast 2}\ast f_r(x)}{g_u(x)}= 2. 
\]
Thus, \eqref{liminf:gu:fr} and Lemma \ref{lem:a2-1} $(\mathrm{iii})$ yield 
\begin{align*}
 2 &= \limsup_{x\to\infty} \frac{g_u^{\ast 2} \ast f_r(x)}{g_u(x)} \\
   &\ge \limsup_{x\to\infty} \frac{g_u^{\ast 2}(x)}{g_u(x)} \liminf_{x\to \infty} \frac{g_u^{\ast 2} \ast f_r(x)}{g_u^{\ast 2}(x)} \\
   &\ge \limsup_{x\to\infty} \frac{g_u^{\ast 2}(x)}{g_u(x)},
\end{align*}
so that $g_u^{\ast 2}(x)/g_u(x) \to 2$ as $x\to\infty$. 
This implies $g_u\in\cals_+$ and $g_1\in \cals_+$ holds by Lemma \ref{lem:a3-2}. \\

\noindent
{\bf $(\mathrm{ii})$ implies $(\mathrm{i})$ and $(\mathrm{iii})$ under the condition that $g$ is al.d.} \\ 
We make use of $g_u$ and $f_u$ in the part:{\bf $(\mathrm{iii})$ implies $(\mathrm{ii})$}. 
By definition $g \in \cals_+$ implies $g_u\in \cals_+$ and $g(x)\sim g_u(x) \Lambda_u$. 
Moreover, $g_u$ is bounded and al.d. 
Hence, by Theorem \ref{thm:compundpoi} 
$f_u\in\cals_+$ and $f_u$ is al.d., and indeed $\lim_{x\to\infty} f_u(x)/g_u(x)=\Lambda_u/(1-e^{-\Lambda_u})$.  
In view of the relation 
\[
f(x)=\wt f_u \ast f_r(x)=e^{-\Lambda_u}f_r(x)+(1-e^{-\Lambda_u})f_r\ast f_u(x), 
\] 
By Lemma \ref{lem:pf:main:gamma:moment}, 
$f_r(x)=o(e^{-\gamma x})$, and so $f_r(x)=o(f_u(x))$. Therefore,
Proposition \ref{prop:a3-3} applied to $f_r\ast f_u$, yields $f(x)\sim (1-e^{-\Lambda_u})f_u(x)$. 
Thus $(\mathrm{i})$ follows from Lemmas \ref{lem:a3-2} and \ref{lem:asymp:eqive:ald}.
For $(\mathrm{iii})$, we observe 
\[
 \lim_{x\to\infty}\frac{f(x)}{g_1(x)}= \lim_{x\to\infty} \frac{f(x)}{f_u(x)} \frac{f_u(x)}{g_u(x)} \frac{g_u(x)}{g_1(x)}=\nu((1,\infty)). 
\] 
\end{proof}

\subsection{Proof of Theorem \ref{theorem:ID:abs}}
We need the following lemma, which characterizes the tail of the density $f_r$ of $\mu_r$ in \eqref{def:mua}.  
\begin{lemma}
\label{lem:abs:int:chf}
 Suppose that $\int_{-\infty}^\infty |\wh \mu_+(z)|dz<\infty $ and then for sufficiently large $c_1$ of \eqref{def:mua}
the density $f_r$ of $\mu_r$ satisfies $\lim_{x\to\infty}e^{\gamma x}f_r(x)=0$ for any $\gamma>0$. 
\end{lemma}
\begin{proof}
We prepare the spectrally positive version $\mu_{r+}$ of $\mu_r$ by 
\[
\wh \mu_{r+}(z) = \exp \Big\{
\int_{0}^{c_1} (e^{izy}-1-izy {\bf 1}_{\{0<y\le 1\}})g(y)dy
\Big\}.
\]
First we see $\int_{-\infty}^\infty |\wh \mu_+(z)|dz<\infty \Leftrightarrow \int_{-\infty}^\infty |\wh \mu_{r+}(z)|dz<\infty$, but 
the proof is only a reproduction of that for Lemma 10 $(\mathrm{i})$ of \cite{Watanabe:2020} and we omit it. 
We show that $\int_{-\infty}^\infty |\wh \mu_{r+}(z)|dz<\infty$ implies 
that its density satisfies 
$f_{r+}(x)=o(e^{-\gamma x})$. Although this part is again quite similar to that of 
Lemma 10 $(\mathrm{ii})$ of \cite{Watanabe:2020}, since we treat a spectrally positive case, we briefly state the outline. 
Because $\int_{1}^{c_1} e^{\gamma x}g(x)dx <\infty$, by \cite[Theorem 25.3]{sato:1999}, we obtain 
$C_r:=\int_{-\infty}^\infty e^{\gamma x}f_{r+}(x)dx<\infty$. Thus, we may define the exponential tilt $\mu_{r+}^\gamma$ on $\R$ as 
$\mu_{r+}^\gamma(dx)=C_r^{-1}e^{\gamma x} f_{r+}(x)dx $. Then due to \cite[Theorem 3.9]{Kuprianou:2006} (cf. \cite[Ex. 33.15]{sato:1999} 
and \cite[Lemma 7]{Watanabe:2020}), $\mu_{r+}^\gamma$ still belongs to \idr\ 
given by \eqref{def:chf:idr} with the L\'evy-Khintchine triplet 
\[
 a= \int_{0}^1(e^{\gamma x}-1)x g(x)dx,\quad b=0 \quad \text{and}\quad \nu(dx)= {\bf 1}_{\{0 < x \le c_1\}}e^{\gamma x} g(x)dx. 
\]
Now observe that 
\[
 |\wh \mu_{r+}^\gamma(z)|=|\wh \mu_{r+}(z)|\exp\Big\{ 
\int_{0}^{c_1}(\cos (zx)-1)(e^{\gamma x}-1)\nu(dx) \Big\}\le |\wh \mu_{r+}(z)|. 
\]
Hence $\wh \mu_{r+}^\gamma$ is absolutely integrable, and the Riemann-Lebesgue lemma implies $f_{r+}(x)=o(e^{-\gamma x})$.
Finally, since $\mu_r=\mu_{r+}\ast \mu_-$ with $\wh \mu_-(z):=\wh \mu(z)/\wh \mu_+(z)$, and $\mu_{r+}$ has a bounded continuous density 
\[
 e^{\gamma x} f_r(x)= \int_{-\infty}^\infty e^{\gamma (x-y)} f_{r+}(x-y) e^{\gamma y} \mu_-(dy) \le c \int_{-\infty}^\infty 
e^{\gamma y}
\mu_-(dx)<\infty. 
\]
Thus we have $f_r(x)=o(e^{-\gamma x})$. 
\end{proof}
\begin{proof}[Proof of Theorem \ref{theorem:ID:abs}]
Notice that the conditions of Theorem \ref{theorem:ID:abs} includes the conditions other than $\nu(\R)=\infty$ and \eqref{exp:limit:trancated}
of Theorem \ref{theorem:ID}, where the former condition is used only for absolutely continuously of $\mu$. 
In view of the proof of Lemma \ref{lem:pf:main:gamma:moment}, the condition \eqref{exp:limit:trancated} is used only for 
deriving $f_r(x)=o(e^{-\gamma x})$. Thus Lemma \ref{lem:pf:main:gamma:moment} holds under the conditions of Theorem 
\ref{theorem:ID:abs} by Lemma \ref{lem:abs:int:chf}. 
Moreover, in the proof of Theorem \ref{theorem:ID}, \eqref{exp:limit:trancated} appears implicitly only 
through the fact: $f_r(x)=o(e^{-\gamma x})$. Thus, we could reuse the proof of Theorem \ref{theorem:ID} for that of Theorem \ref{theorem:ID:abs}. 
\end{proof}

\appendix

\section{Proofs for remaining results in Sections \ref{sec:main:results} and \ref{sec:cp}}
\label{append:proofs}
The proofs for the versions of the convolution root are given in this section. 
Although these are not directly related for the main results, 
they have their own interest. 
Throughout this section let $c$ be a positive constant whose value may differ depending on context. 

\begin{proof}[Proof of {\rm Theorem \ref{prop:a3-5}}]
{\bf The case $\wt f^{\ast k}$ are a.n.i.}  We start by deriving 
\begin{align}
 \liminf_{x\to\infty} \wt f^{\ast n}(x)/\wt f(x) \ge n.  \label{pf:convolution:root:liminf}
\end{align}
Let $z<0$. We take a uniform constant $x_0>0$ such that for $x\ge x_0$, $\wt f^{\ast k}(x),\,1\le k \le n-1$ 
satisfy the condition \eqref{eq:def:ani}. 
Then for $x>x_0$ sufficiently large, we have  
\begin{align*}
 \wt f^{\ast n}(x+(n-1)z) 
&= \int_{(n-1)z/2}^\infty \wt f^{\ast(n-1)}(x/2+(n-1)z-y) \wt f(y+x/2)dy \\
&\qquad + \int_{(n-1)z/2}^\infty \wt f(x/2+(n-1)z-y) \wt f^{\ast(n-1)}(y+x/2)dy\\
&\ge \int_{x_0}^{x/2} \wt f^{\ast(n-1)}(x/2+(n-1)z-y) \wt f(y+x/2)dy \\
&\qquad + \int_{x_0}^{x/2+(n-2)z} \wt f(x/2+(n-1)z-y) \wt f^{\ast(n-1)}(y+x/2)dy\\
& \ge \wt f(x) (1-\vep_1^x)\int_{(n-1)z}^{x/2+(n-1)z-x_0} \wt f^{\ast(n-1)}(y)dy \\
&\qquad + \wt f^{\ast(n-1)}(x+(n-2)z) (1-\delta_{n-1}^x) \int_z^{x/2+(n-1)z-x_0}\wt f(y)dy \\
&=: I_{n-1}^x(z)\wt f(x)+J_{n-1}^x(z)\wt f^{\ast(n-1)}(x+(n-2)z), 
\end{align*}
where $\vep_1^x,\delta_{n-1}^x \in(0,1)$ are small constants such that $\vep_1^x,\delta_{n-1}^x\to 0$ as $x\to \infty$.
We further take small constants 
$\vep_k^x,\delta_k^x \in(0,1),\,k=2,\ldots,n-2$ such that $\vep_k^x,\delta_k^x\to 0$ as $x\to \infty$, 
and we successively apply the inequality and reach 
\begin{align}
\label{ineq:reccursion}
 \wt f^{\ast n}(x+(n-1)z) \ge \wt f(x) \sum_{k=1}^n I_{n-k}^x (z)\prod_{\ell=1}^{k-1} J_{n-\ell}^x (z), 
\end{align}
where $I_0^x(z):=1$ and $\prod_{\ell=1}^0 J_{n-\ell}^x(z):=1$. Since all $I_k^x(z),\,J_k^x(z),\,1\le k\le n-1$ satisfy 
\begin{align*}
 \lim_{z\to -\infty}\lim_{x\to \infty} I_k^x(z)= \lim_{z\to -\infty} \lim_{x\to \infty} J_k^x(z)=1, 
\end{align*}
by $\wt f^{\ast n} \in \call$ we have 
\begin{align*}
 \liminf_{x\to \infty} \frac{\wt f^{\ast n}(x)}{\wt f(x)} = \lim_{z\to-\infty} \liminf_{x\to\infty} 
\frac{\wt f^{\ast n}(x+(n-1)z)}{ \wt f(x)} \ge \lim_{z\to-\infty} \liminf_{x\to\infty} \sum_{k=1}^n I_{n-k}^x (z)\prod_{\ell=1}^{k-1} J_{n-\ell}^x (z)=n . 
\end{align*} 

Next we prove the other direction, 
\begin{align}
 \limsup_{x\to\infty} \wt f^{\ast n}(x) / \wt f(x) \le n. \label{pf:convolution:root:limisup}
\end{align}
This time we take $z>0$ and observe that 
\begin{align*}
  \wt f^{\ast n}(x+(n-1)z) 
&= \int_{-\infty}^z \wt f^{\ast(n-1)}(x+(n-1)z-y) \wt f(y)dy \\
&\qquad + \Big(\int_{-\infty}^{(n-1)z} + \underbrace{\int_{(n-1)z}^{x+(n-2)z} 
\Big)\wt f(x+(n-1)z-y) \wt f^{\ast(n-1)}(y)dy}_{=:\bar J_{n-1}^x (z)}\\
& \le \wt f^{\ast (n-1)}(x+(n-2)z)(1+\bar \vep_{n-1}^x) \int_{-\infty}^z \wt f(y)dy \\
&\qquad +\wt f(x) (1+\bar \vep_1^x) \int_{-\infty}^{(n-1)z}\wt f^{\ast (n-1)}(y)dy +\bar J_{n-1}^x (z) \\
&=: \wt f^{\ast (n-1)}(x+(n-2)z) \bar I_1^x(z)+ \wt f(x) \bar I_{n-1}^x (z) + \bar J_{n-1}^x (z), 
\end{align*}
where $\bar \vep_k^x\in(0,1),\,k=1,\ldots,n-1$ are small constants such that $\bar \vep_k^x \to 0$ as $x\to \infty$. 
We successively apply the inequality above and obtain 
\begin{align}
\label{ineq:reccursion2} 
 \wt f^{\ast n}(x+(n-1)z) &\le \Big\{
\sum_{k=1}^{n-2}\bar I_{n-k}^x(z)\big(\bar I_1^x(z)\big)^{k-1}+2\big(
\bar I_1^x(z)\big)^{n-1} 
\Big\} \wt f(x) \\
&\quad + \sum_{k=1}^{n-1} \bar J_{n-k}^x(z) \big(\bar I_1^x(z)\big)^{k-1}, \nonumber 
\end{align}
where $\big(\bar I_1^x(z)\big)^0:=1$. 
We introduce a non-decreasing function $0<\alpha(x)<x/2$ such that 
$\wt f^{\ast n}$ is $\alpha$-insensitive (Lemma \ref{lem:a2-1} $(\mathrm{i})$) and put $(n-1)z=\alpha(x)$, so that 
$z=\alpha'(x):=\alpha(x)/(n-1)$. Notice that $\alpha'$ is again an insensitive function for $\wt f^{\ast n}$ (cf. \cite[p.20]{Foss:Korshunov:Zachary:2013}). 
Obviously 
\[
\lim_{x\to\infty} \bar I_k^x(\alpha'(x)) = \lim_{x\to\infty}(1+\bar \vep_{n-k}^x) 
\int_{-\infty}^{k\alpha'(x)} \wt f^{\ast k}(y)dy =1. 
\]
We show that 
\begin{align}
\label{def:Jkx:smaller:fn}
 \bar J_k^x (\alpha'(x))=o(\wt f^{\ast n}(x))\quad \text{as}\quad x\to\infty. 
\end{align}
First we see that 
\begin{align}
\label{limsup:fk:fn}
 \limsup_{x\to\infty} \wt f^{\ast k}(x)/\wt f^{\ast n}(x) \le 1\quad \text{for}\quad 1\le k\le n-1. 
\end{align}
For $v>0$, we write 
\begin{align*}
 \wt f^{\ast n}(x-v) 
&\ge \int_{-v}^v \wt f^{\ast k}(x-v-y)\wt f^{\ast (n-k)}(y)dy \ge \inf_{z\in [x-2v,x]}\wt f^{\ast k}(z) \int_{-v}^v \wt f^{\ast(n-k)}(y)dy. 
\end{align*}
Since $\wt f^{\ast k}$ is a.n.i., it follows that 
\begin{align*}
 \liminf_{x\to\infty}\wt f^{\ast n}(x)/\wt f^{\ast k}(x) = \lim_{v\to\infty} \liminf_{x\to\infty} \frac{\wt f^{\ast n}(x)}{\wt f^{\ast n}(x-v)} 
\frac{\wt f^{\ast n}(x-v)}{\wt f^{\ast k}(x)} \ge \lim_{v\to \infty}\int_{-v}^v \wt f^{\ast(n-k)}(y)dy =1. 
\end{align*}
We return to $\bar J_k^x(z)$ and observe that for sufficiently large $z>0$ 
\begin{align*}
 \bar J_k^x(z) &= \int_z^x \wt f(x+z-y)\wt f^{\ast k}(y+(k-1)z)dy \le c \int_{z/2}^{x-z/2} \wt f^{\ast n}(x-y)\wt f^{\ast n}(y)dy, 
\end{align*}
where we use \eqref{limsup:fk:fn} and the a.n.i. property of $\wt f^{\ast n}$. 
Then recalling that $\alpha'$ is an insensitive function of $f^{\ast n}$ and so is $\alpha'/2$, we have by 
Lemma \ref{lem:a2-1} $(\mathrm{v})$:\eqref{condi:subexponential:real}, that 
\begin{align*}
 \bar J_k^x(\alpha'(x)) &\le c \int_{\alpha'(x)/2}^{x-\alpha'(x)/2} \wt f^{\ast n}(x-y) \wt f^{\ast n}(y)dy = o(\wt f^{\ast n}(x)). 
\end{align*}
Now in view of \eqref{ineq:reccursion2} 
putting $z=\alpha'(x)$, we have 
\begin{align*}
1 &\le  \liminf_{x\to\infty} \frac{\wt f(x)}{\wt f^{\ast n}(x+\alpha(x))}\Big[
\sum_{k=1}^{n-2}\bar I_{n-k}^x(\alpha'(x)) \big\{\bar I_1^x(\alpha'(x))\big\}^{k-1}+2
\big\{\bar I_1^x(\alpha'(x)) \big\}^{n-1}
\Big] \\
&\quad + \limsup_{x\to\infty} \frac{\sum_{k=1}^{n-1} \bar J_{n-k}^x(\alpha'(x)) \big\{\bar I_1^x(\alpha'(x))\big\}^{k-1}}{f^{\ast n}(x+\alpha(x))} \\
 & \le  \liminf_{x\to\infty} \frac{\wt f(x)}{\wt f^{\ast n}(x)} \lim_{x\to\infty}
\Big[
\sum_{k=1}^{n-2}\bar I_{n-k}^x(\alpha'(x)) \big\{\bar I_1^x(\alpha'(x))\big\}^{k-1}+2
\big\{\bar I_1^x(\alpha'(x)) \big\}^{n-1}
\Big] 
+ o(1) \\ 
 & = n \cdot \liminf_{x\to\infty} \frac{\wt f(x)}{\wt f^{\ast n}(x)}, 
\end{align*}
where we use \eqref{def:Jkx:smaller:fn} in the second step. Thus we obtain \eqref{pf:convolution:root:limisup}. 
Finally we apply Lemma \ref{lem:a3-2} to the fact that $\lim_{x\to\infty} \wt f^{\ast n}(x)/\wt f(x)=n$ and obtain the result. 

\noindent
{\bf The case $f\in\call$.} Since $\wt f^{\ast k}\in \call$ for $k=1,\ldots, n$ (cf. Lemma \ref{lem:a2-1} $(\mathrm{iv})$), we may take a single 
non-decreasing function $0<\alpha(x) < x/2$ such that $\wt f^{\ast k}$ is $\alpha$-insensitive (Lemma \ref{lem:a2-1} $(\mathrm{ii})$).
We decompose the integral form for $\wt f^{\ast n}=\wt f^{\ast (n-1)}\ast \wt f$ and write for $x>0$
\begin{align*}
1 &= \Big(
\int_{-\infty}^{-\alpha(x)}+\int_{-\alpha(x)}^{\alpha(x)}+\int_{\alpha(x)}^{x-\alpha(x)} +\int_{x-\alpha(x)}^{x+\alpha(x)}+\int_{x+\alpha(x)}^\infty 
\Big) \frac{\wt f^{\ast (n-1)}(x-y)\wt f(y)}{\wt f^{\ast n}(x)}dy \\
&=:I_1(x)+\cdots + I_5(x). 
\end{align*} 
We start with $I_2(x)$. Since $\wt f^{\ast (n-1)} \in \call$ and $\wt f^{\ast n}\in \cals$, we have for $|y|\le \alpha(x)$ 
\begin{align}
 \limsup_{x\to\infty} \frac{\wt f^{\ast (n-1)}(x-y)}{\wt f^{\ast n}(x)}=
 \limsup_{x\to\infty} \frac{\wt f^{\ast (n-1)}(x)}{\wt f^{\ast n}(x)}= \limsup_{x \to \infty}  
\frac{\wt f^{\ast (n-1)}(x)}{\wt f^{\ast n(n-1)}(x)}
\frac{\wt f^{\ast n(n-1)}(x)}{\wt f^{\ast n}(x)} \le 1-n^{-1}. \label{thm:convo:root}
\end{align}
Recalling the form $\wt f(x)= p\delta(x)+q f(x)$, we have by the property of $\delta$ that 
\begin{align*}
 \limsup_{x\to\infty} I_2(x) 
&\le p \limsup_{x\to\infty} \frac{\wt f^{\ast(n-1)}(x)}{\wt f^{\ast n}(x)} + q \limsup_{x\to\infty}
\int_{-\alpha(x)}^{\alpha(x)}  \frac{\wt f^{\ast (n-1)}(x-y)}{\wt f^{\ast n}(x)}
f(y)dy \\
& \le p(1-n^{-1})+ q (1-n^{-1}) \int_{-\infty}^\infty f(y)dy =1-n^{-1},
\end{align*}
where in the second term, we use Fatou's lemma, which is possible since the delta function is not involved. 
It follows from \eqref{thm:convo:root} and a.n.i. property of $\wt f^{\ast n}$ that 
\[
 I_1(x)= \int_{-\infty}^{-\alpha(x)} \frac{\wt f^{\ast(n-1)}(x-y)}{\wt f^{\ast n}(x-y)} \frac{\wt f^{\ast n}(x-y)}{\wt f^{\ast n}(x)} \wt f(y) dy
\le c \int_{-\infty}^{-\alpha(x)} \wt f(y)dy \to 0 
\]
as $x\to \infty$. 
For $I_3(x)$, noticing an expression 
\[
 I_3(x)= \int_{\alpha(x)}^{x-\alpha(x)} \frac{\wt f^{\ast (n-1)}(x-y)}{\wt f^{\ast n}(x-y)} \frac{\wt f(y)}{\wt f^{\ast n}(y)} 
\frac{\wt f^{\ast n}(x-y)}{\wt f^{\ast n}(x)}{\wt f^{\ast n}(y)}dy,  
\]
we apply \eqref{thm:convo:root} and Lemma \ref{lem:a2-1} $(\mathrm{iii})$ respectively to the first and the second 
terms of the integrand. Then Lemma \ref{lem:a2-1} $(\mathrm{v})$ yields $\lim_{x\to\infty} I_3(x)=0$.
By the dominated convergence 
\[
\liminf_{x\to \infty} I_4(x)=
\liminf_{x\to \infty}\frac{\wt f(x)}{\wt f^{\ast n}(x)} 
\lim_{x\to\infty}\int_{-\alpha(x)}^{\alpha(x)} \frac{\wt f(x-y)}{\wt f(x)} \wt f^{\ast (n-1)}(y)dy
= \liminf_{x\to \infty}\frac{\wt f(x)}{\wt f^{\ast n}(x)}.  
\]
Here to apply the dominated convergence avoiding $\delta$, if necessary, write 
\[
 \wt f^{\ast (n-1)}(y)= (p\delta+q f)^{\ast (n-1)}(y) =\sum_{k=1}^{n-1}{}_{n-1}C_{k} f^{\ast k}(y)q^kp^{n-1-k}+p^{n-1}\delta(y)
\]
and take a similar approach as for $I_2(x)$. 

Finally we apply Lemma \ref{lem:a2-1} $(\mathrm{iii})$ and the a.n.i. property of $\wt f^{\ast n}$ to $I_5$, and obtain 
\begin{align*}
 I_5(x) &= \int_{-\infty}^{-\alpha(x)} \frac{\wt f(x-y)}{\wt f^{\ast n}(x-y)} \frac{\wt f^{\ast n}(x-y)}{\wt f^{\ast n}(x)} \wt f^{\ast (n-1)}(y)dy.  \\
& \le n^{-1} \int_{-\infty}^{-\alpha(x)} \wt f^{\ast (n-1)}(y)dy \to 0\quad \text{as}\ x\to\infty. 
\end{align*}
Now correcting above bounds, we reach 
\begin{align*}
 1=\liminf_{x\to\infty}\big(
\sum_{i=1}^5 I_i(x)
\big) &\le \liminf_{x\to\infty} I_4(x)+\limsup_{x\to\infty} \sum_{i\neq 4}I_i(x) =\liminf_{x\to\infty} \wt f(x)/\wt f^{\ast n}(x)+1-n^{-1},
\end{align*}
which is equal to \eqref{pf:convolution:root:limisup}. 
By Lemma \ref{lem:a2-1} $(\mathrm{iii})$, we obtain 
\[
 \lim_{x\to\infty} \wt f^{\ast n}(x)/\wt f(x) = n. 
\] 
\end{proof}

\begin{proof}[Proof of Proposition \ref{prop:convroot:gene}]
Observe that 
\begin{align}
\label{eq:pf:convroot}
 \wt f^{\ast n}(x)=\sum_{k=0}^n \binom{n}{k}(1-p)^k p^{n-k} f^{\ast k}(x),
\end{align} 
where $f^{\ast 0}(x)=\delta(x)$ and define 
\begin{align*}
 \liminf_{x\to\infty} f(x)/\wt f^{\ast n}(x)= \underline{C}\quad \text{and}\quad \limsup_{x\to\infty} f(x)/\wt f^{\ast n}(x) = \ov{C},
\end{align*}
which are well-defined since $\wt f^{\ast n}(x) \ge n(1-p)p^{n-1}f(x)$. We show by induction that 
\begin{align}
\label{pr:df:liminfsup}
 \liminf_{x\to\infty} f^{\ast k}(x)/\wt f^{\ast n}(x) \ge k\underline{C}\quad \text{and}\quad 
 \limsup_{x\to\infty} f^{\ast k}(x)/\wt f^{\ast n}(x) \textcolor{red}{\le} k\ov{C}
\end{align}
hold. Since the proof for the $\limsup$ part is similar, we only consider the $\liminf$ part. Suppose that \eqref{pr:df:liminfsup} holds 
with $k-1,\,k\ge 2$ and consider 
\begin{align*}
 \frac{f^{\ast k}(x)}{\wt f^{\ast n}(x)} &= \Big(
\int_{-\alpha(x)}^{\alpha(x)} + \int_{x-\alpha(x)}^{x+\alpha(x)}+ \int_{\alpha(x)}^{x-\alpha(x)}+
\int_{-\infty}^{-\alpha(x)} + \int_{x+\alpha(x)}^\infty
\Big) \frac{f(x-y) f^{\ast (k-1)}(y)}{\wt f^{\ast n}(x)}dy \\
&=: I_1(x)+I_2(x)+I_3(x)+I_4(x)+I_5(x),
\end{align*}
where $\alpha$ is an insensitive function for $\wt f^{\ast n}$. By Fatou's lemma and $\wt f^{\ast n}\in \call$,  
\begin{align}
\label{ineq1:pf:I1}
 \liminf_{x\to \infty} I_1(x) &=\liminf_{x\to\infty} \int_{-\alpha(x)}^{\alpha(x)}
\frac{f(x-y)}{\wt f^{\ast n}(x-y)} \frac{\wt f^{\ast n}(x-y)}{\wt f^{\ast n}(x)} f^{\ast (k-1)}(y)dy  \\
& \ge \int_{-\infty}^{\infty} \liminf_{x\to\infty} \frac{f(x-y)}{\wt f^{\ast n}(x-y)} {\bf 1}_{\{y\in [-\alpha(x),\alpha(x)]\}} f^{\ast (k-1)}(y) dy \ge 
\underline{C}  \nonumber
\end{align}
and by the induction hypothesis
\begin{align}
\label{ineq1:pf:I2}
  \liminf_{x\to \infty} I_2(x) &=\liminf_{x\to\infty} \int_{-\alpha(x)}^{\alpha(x)}
\frac{f^{\ast (k-1)}(x-y)}{\wt f^{\ast n}(x-y)} \frac{\wt f^{\ast n}(x-y)}{\wt f^{\ast n}(x)} f(y)dy  \\
& \ge \int_{-\infty}^{\infty} \liminf_{x\to\infty} \frac{f^{\ast (k-1)}(x-y)}{\wt f^{\ast n}(x-y)} {\bf 1}_{\{y\in [-\alpha(x),\alpha(x)]\}} f(y) dy \ge 
(k-1)\underline{C}.  \nonumber
\end{align}
Moreover, since $\wt f^{\ast n}(x) \ge n(1-p)p^{n-1}f(x)+\binom{n}{k-1}(1-p)^{k-1}p^{n-k+1}f^{\ast (k-1)}(x)$ for all $x\in \R$,  
\begin{align}
\label{ineq1:pf:I3}
 \limsup_{x\to \infty}I_3(x) \le c \limsup_{x\to\infty} \int_{\alpha(x)}^{x-\alpha(x)}
 \frac{\wt f^{\ast n}(x-y) \wt f^{\ast n}(y)}{\wt f^{\ast n}(x)}dy =0 
\end{align}
by Lemma \ref{lem:a2-1} $(\mathrm{v})$, while by exactly the same logic, 
\begin{align}
\label{ineq1:pf:I45}
 \limsup_{x\to\infty} (I_4(x)+I_5(x)) \le c \limsup_{x\to\infty}\int^{-\alpha(x)}_{-\infty} 
\frac{\wt f^{\ast n}(x-y) \wt f^{\ast n}(y)}{\wt f^{\ast n}(x)}dy =0.  
\end{align}
Now collecting \eqref{ineq1:pf:I1}-\eqref{ineq1:pf:I45} we obtain \eqref{pr:df:liminfsup}. 
Then recalling \eqref{eq:pf:convroot} we observe that 
\begin{align*}
 1 &\ge \limsup_{x\to\infty} n (1-p) p^{n-1} f(x)/\wt f^{\ast n}(x)\\
& \quad + \liminf_{x\to \infty}
\sum_{k= 2}^n \binom{n}{k}(1-p)^k p^{n-k} f^{\ast k}(x)/\wt f^{\ast n}(x) \\
&\ge  n (1-p) p^{n-1} \ov C+ \sum_{k=2}^n \binom{n}{k} (1-p)^k p^{n-k} k \underline{C} \\
& = n(1-p)\big\{p^{n-1} \ov C+ (1-p^{n-1})\underline{C}\big\}
\end{align*}
and 
\begin{align*}
 1 &\le \liminf_{x\to\infty} n(1-p)p^{n-1} f(x)/\wt f^{\ast n}(x) \\
&+ \limsup_{x\to \infty}\sum_{k=2}^n 
\binom{n}{k}(1-p)^k p^{n-k} f^{\ast k}(x)/\wt f^{\ast n}(x) \\
& \le n(1-p)\big\{ p^{n-1} \underline{C}+ (1-p^{n-1})\ov C\big\}, 
\end{align*}
which together yield
\[
 0\le n(1-p) (1-2p^{n-1})(\ov C-\underline{C}).
\] 
From the condition $2^{-1/(n-1)}< p$, $\ov C =\underline{C}$ should hold. 
Then, since $f(x)\le c \wt f^{\ast n}(x)$ for all $x \in \R$, $\wt f\in \cals$ follows from Lemma \ref{lem:a3-2}. 
\end{proof}

\begin{proof}[Proof of Corollary \ref{cor:steutel}]
Take $\alpha=n^{-1},n\in \N$, so that $(\wt f^{\ast \alpha})^{\ast n}=\wt f$ and the Poisson parameter of 
$\wt f^{\ast \alpha}$ is $\lambda/n$. The coefficient of $\delta$ of $\wt f^{\ast \alpha}$ satisfies 
$e^{-\lambda/n} > 2^{-1/n}>2^{-1/(n-1)}$. Thus by Proposition \ref{prop:convroot:gene}, $\wt f^{\ast \alpha}\in \cals$ and 
$\lim_{x\to\infty} \wt f^{\ast \alpha}(x)/f(x)=\alpha$. This implies that the result holds for any rational $\alpha>0$.
\end{proof}

\noindent {\bf Acknowledgments}
The earlier version of results in this paper has been presented 
at the annual workshop “Infinitely divisible processes and related topics” held in Nov. 2021. 
The results have been significantly improved after the workshop and the author acknowledges the comments and  
the hosts in the workshop. The author is grateful to Toshiro Watanabe for careful reading and 
all comments and discussions about subexponentiality. 
Particularly, his suggestion of the relation between the local subexponentiality and the topic yielded 
substantial improvement of the main theorem.  
The author's research is partly supported by the JSPS Grant-in-Aid for Scientific Research C
(19K11868).

\end{document}